\documentclass[12pt]{article}
\usepackage[english]{babel}
\usepackage{pdfsync}
\usepackage{dynkin-diagrams}
\usepackage{paralist}
\usepackage{amssymb,amsmath,amsthm}
\usepackage{xcolor}
\usepackage{tikz,graphics}
\usepackage{xy}
\usepackage{calrsfs}            
\usepackage{mdwlist}
\usepackage{comment}
\parskip\medskipamount
\usepackage[capitalize]{cleveref}
    \crefname{enumi}{}{}
    \Crefname{enumi}{Item}{Items}
    \crefname{equation}{}{}
    \Crefname{equation}{Equation}{Equations}
\newtheorem{theo}{Theorem}
\newtheorem{coro}[theo]{Corollary}
    
\newtheorem{theorem}{Theorem}[section]
\newtheorem{prop}[theorem]{Proposition}
\newtheorem{lemma}[theorem]{Lemma}
\newtheorem{cor}[theorem]{Corollary}

\theoremstyle{definition}
\newtheorem{definition}[theorem]{Definition}
\newtheorem{rem}[theorem]{Remark}
\numberwithin{equation}{section}
\def\<{\langle}
\def\>{\rangle}

\newcommand{\cM}{\mathcal{M}}

\newcommand{\cI}{\mathcal{I}}
\newcommand{\proj}{\mathsf{proj}}
\newcommand{\kar}{\mathsf{char}\,}

\newcommand{\disp}{\mathsf{disp}}
\newcommand{\Cl}{\mathsf{Cl}}

\newcommand{\JJoin}{\!\!\Join\!\!}

\newcommand{\PG}{\mathsf{PG}}

\renewcommand{\H}{\mathbb{H}}
\newcommand{\K}{\mathbb{K}}
\renewcommand{\L}{\mathbb{L}}

\newcommand{\A}{\mathbb{A}}
\newcommand{\B}{\mathbb{B}}

\newcommand{\cL}{\mathcal{L}}

\newcommand{\PSL}{\mathsf{PSL}}

\newcommand{\Res}{\mathsf{Res}}
\newcommand{\Sym}{\mathsf{Sym}}
\newcommand{\Norm}{\mathsf{N}}

\newcommand{\pperp}{\perp\hspace{-0.15cm}\perp}

\def\sE{\mathsf{E}}
\def\sA{\mathsf{A}}
\def\ZZ{\mathbb{Z}}
\def\sD{\mathsf{D}}

\def\sF{\mathsf{F}}
\def\sG{\mathsf{G}}
\def\su{\mathsf{u}}
\def\KK{\mathbb{K}}
\usepackage[normalem]{ulem}
\def\OOpp{\mathsf{Opp}}

\def\FF{\mathbb{F}}

\let\OLDthebibliography\thebibliography
\renewcommand\thebibliography[1]{
  \OLDthebibliography{#1}
  \setlength{\parskip}{0pt}
  \setlength{\itemsep}{0pt plus 0.3ex}
}

\usepackage{lineno}
\begin{document}

\author{J.~Parkinson\thanks{
The first author is supported by the Australian Research Council Discovery Project~DP200100712.}  \and H.~Van Maldeghem\thanks{The second author is supported by the Fund for Scientific Research, Flanders, through the project G023121N}}
\title{Automorphisms and opposition in spherical buildings of exceptional type, V: The $\mathsf{E_8}$ case}
\date{\today}
\maketitle

\begin{abstract}
An automorphism of a spherical building is called \textit{domestic} if it maps no chamber to an opposite chamber. In previous work the classification of domestic automorphisms in large spherical buildings of types $\sF_4$, $\sE_6$, and $\sE_7$ have been obtained, and in the present paper we complete the classification of domestic automorphisms of large spherical buildings of exceptional type of rank at least~$3$ by classifying such automorphisms in the $\sE_8$ case. Applications of this classification are provided, including Density Theorems showing that each conjugacy class in a group acting strongly transitively on a spherical building intersects a very small number of $B$-cosets, with $B$ the stabiliser of a fixed choice of chamber.  
\end{abstract}

\tableofcontents 

\section{Introduction}

The \textit{opposite geometry} of an automorphism~$\theta$ of a spherical building~$\Omega$ is the set $\OOpp(\theta)$ consisting of those elements mapped to opposite elements by~$\theta$. Recently a detailed theory of this geometry has been developed (see, for example,~\cite{AB:09,Lam-Mal:24,PVM:19,PVMexc,PVMclass}). A starting point is the fundamental result of Abramenko and Brown~\cite[Proposition 4.2]{AB:09}, stating that if $\theta$ is a nontrivial automorphism of a thick spherical building then the opposite geometry of~$\theta$ is necessarily nonempty. Indeed the generic situation is that $\OOpp(\theta)$ is rather large, and typically contains many chambers of the building (\textit{chambers} are the simplices of maximal dimension). The more special situation is when $\OOpp(\theta)$ contains no chamber, in which case $\theta$ is called \textit{domestic}. 

In \cite{PVMexc} we initiated the classification of domestic automorphisms of spherical buildings of exceptional type. In particular we classified all domestic automorphisms of thick buildings of type $\mathsf{E_6}$, and split buildings of types $\sF_4$ and $\sG_2$. In \cite{PVMexc2} we obtained the classification of domestic automorphisms for all Moufang hexagons, and in \cite{Lam-Mal:24} Lambrecht and the second author obtained the classification of domestic automorphisms for all large thick buildings of type~$\sF_4$. In \cite{npvv} we, together with Neyt and Victoor, classified domestic automorphisms of large $\sE_7$ buildings (an irreducible thick spherical building of rank at least $3$ is called \textit{large} if it contains no Fano plane residues, and \textit{small} otherwise).

The purpose of this paper is to complete the classification of domestic automorphisms of large spherical buildings of exceptional type by classifying domestic automorphisms of large buildings of type~$\sE_8$. The fact that makes type $\sE_8$ harder and different from the other exceptional types, is that no ``simpler'' geometry than the long root subgroup geometry exists for buildings of type $\sE_8$. For types $\sE_6$ and $\sE_7$, for example, we have the so-called minuscule geometries (with below notation, these are the geometries $\mathsf{E}_{6,1}(\K)$ and $\mathsf{E}_{7,7}(\K)$).  Our new technique relies on a more systematic use of the \emph{equator geometries} (for a definition, see below), which allows one to use the classification of domestic automorphisms of spherical buildings of lower rank. A typical example is  the proof of  \cref{classIisdom}, which is one of the crucial results of the present paper. 

By Tits' classification of spherical buildings~\cite{Tits:74}, a large building of type $\sE_8$ is any building $\sE_8(\KK)$ over a field $\KK$ with at least $3$ elements (the unique small $\sE_8$ building corresponds to the field with $2$ elements). By the main result of \cite{PVM:19}, each automorphism $\theta$ of a large spherical building is \textit{capped}, meaning that it satisfies the following property: If $\theta$ maps a flag of type $J_1$ to an opposite flag, and another flag of type $J_2$ to an opposite flag, then $\theta$ maps a flag of type $J_1\cup J_2$ to an opposite flag. Hence in a large spherical building, in order to know the types of all flags mapped to an opposite it suffices to know the types of the minimal ones. Since these minimal types are orbits of the induced action of the automorphism on the Dynkin (or Coxeter) diagram, and since that action is always trivial in the case of $\mathsf{E_8}$, it suffices to know the types of the vertices mapped to an opposite. Encircling those types on the Coxeter diagram gives the \emph{opposition diagram} of the automorphism. In \cite{PVM:19}, all possible opposition diagrams are classified, and the list for $\mathsf{E_8}$ is given in Figure~\ref{fig:Dynkin}. 

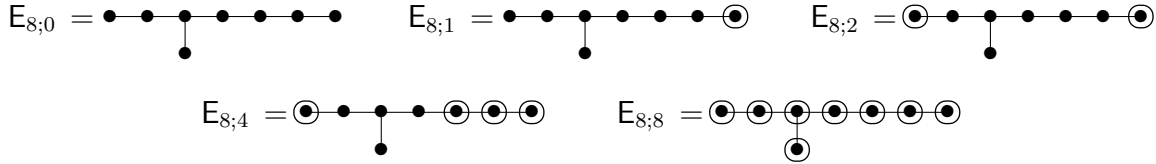
\begin{figure}[h!]
\begin{center}
$\sE_{8;0}\,=\,$\begin{tikzpicture}[scale=0.5,baseline=-0.75ex]
\node [inner sep=0.8pt,outer sep=0.8pt] at (-2,0) (1) {$\bullet$};
\node [inner sep=0.8pt,outer sep=0.8pt] at (-1,0) (3) {$\bullet$};
\node [inner sep=0.8pt,outer sep=0.8pt] at (0,0) (4) {$\bullet$};
\node [inner sep=0.8pt,outer sep=0.8pt] at (1,0) (5) {$\bullet$};
\node [inner sep=0.8pt,outer sep=0.8pt] at (2,0) (6) {$\bullet$};
\node [inner sep=0.8pt,outer sep=0.8pt] at (3,0) (7) {$\bullet$};
\node [inner sep=0.8pt,outer sep=0.8pt] at (4,0) (8) {$\bullet$};
\node [inner sep=0.8pt,outer sep=0.8pt] at (0,-1) (2) {$\bullet$};
\draw (-2,0)--(4,0);
\draw (0,0)--(0,-1);
\end{tikzpicture}\qquad
$\sE_{8;1}\,=\,$\begin{tikzpicture}[scale=0.5,baseline=-0.75ex]
\node [inner sep=0.8pt,outer sep=0.8pt] at (-2,0) (1) {$\bullet$};
\node [inner sep=0.8pt,outer sep=0.8pt] at (-1,0) (3) {$\bullet$};
\node [inner sep=0.8pt,outer sep=0.8pt] at (0,0) (4) {$\bullet$};
\node [inner sep=0.8pt,outer sep=0.8pt] at (1,0) (5) {$\bullet$};
\node [inner sep=0.8pt,outer sep=0.8pt] at (2,0) (6) {$\bullet$};
\node [inner sep=0.8pt,outer sep=0.8pt] at (3,0) (7) {$\bullet$};
\node [inner sep=0.8pt,outer sep=0.8pt] at (4,0) (8) {$\bullet$};
\node [inner sep=0.8pt,outer sep=0.8pt] at (0,-1) (2) {$\bullet$};
\draw (-2,0)--(4,0);
\draw (0,0)--(0,-1);
\draw [line width=0.5pt,line cap=round,rounded corners] (8.north west)  rectangle (8.south east);
\end{tikzpicture}\qquad
$\sE_{8;2}\,=\,$\begin{tikzpicture}[scale=0.5,baseline=-0.75ex]
\node [inner sep=0.8pt,outer sep=0.8pt] at (-2,0) (1) {$\bullet$};
\node [inner sep=0.8pt,outer sep=0.8pt] at (-1,0) (3) {$\bullet$};
\node [inner sep=0.8pt,outer sep=0.8pt] at (0,0) (4) {$\bullet$};
\node [inner sep=0.8pt,outer sep=0.8pt] at (1,0) (5) {$\bullet$};
\node [inner sep=0.8pt,outer sep=0.8pt] at (2,0) (6) {$\bullet$};
\node [inner sep=0.8pt,outer sep=0.8pt] at (3,0) (7) {$\bullet$};
\node [inner sep=0.8pt,outer sep=0.8pt] at (4,0) (8) {$\bullet$};
\node [inner sep=0.8pt,outer sep=0.8pt] at (0,-1) (2) {$\bullet$};
\draw (-2,0)--(4,0);
\draw (0,0)--(0,-1);
\draw [line width=0.5pt,line cap=round,rounded corners] (1.north west)  rectangle (1.south east);
\draw [line width=0.5pt,line cap=round,rounded corners] (8.north west)  rectangle (8.south east);
\end{tikzpicture}
\end{center}
\begin{center}
$\sE_{8;4}\,=\,$\begin{tikzpicture}[scale=0.5,baseline=-0.75ex]
\node [inner sep=0.8pt,outer sep=0.8pt] at (-2,0) (1) {$\bullet$};
\node [inner sep=0.8pt,outer sep=0.8pt] at (-1,0) (3) {$\bullet$};
\node [inner sep=0.8pt,outer sep=0.8pt] at (0,0) (4) {$\bullet$};
\node [inner sep=0.8pt,outer sep=0.8pt] at (1,0) (5) {$\bullet$};
\node [inner sep=0.8pt,outer sep=0.8pt] at (2,0) (6) {$\bullet$};
\node [inner sep=0.8pt,outer sep=0.8pt] at (3,0) (7) {$\bullet$};
\node [inner sep=0.8pt,outer sep=0.8pt] at (4,0) (8) {$\bullet$};
\node [inner sep=0.8pt,outer sep=0.8pt] at (0,-1) (2) {$\bullet$};
\draw (-2,0)--(4,0);
\draw (0,0)--(0,-1);
\draw [line width=0.5pt,line cap=round,rounded corners] (1.north west)  rectangle (1.south east);
\draw [line width=0.5pt,line cap=round,rounded corners] (6.north west)  rectangle (6.south east);
\draw [line width=0.5pt,line cap=round,rounded corners] (7.north west)  rectangle (7.south east);
\draw [line width=0.5pt,line cap=round,rounded corners] (8.north west)  rectangle (8.south east);
\end{tikzpicture}
\qquad
$\sE_{8;8}\,=\,$\begin{tikzpicture}[scale=0.5,baseline=-0.75ex]
\node [inner sep=0.8pt,outer sep=0.8pt] at (-2,0) (1) {$\bullet$};
\node [inner sep=0.8pt,outer sep=0.8pt] at (-1,0) (3) {$\bullet$};
\node [inner sep=0.8pt,outer sep=0.8pt] at (0,0) (4) {$\bullet$};
\node [inner sep=0.8pt,outer sep=0.8pt] at (1,0) (5) {$\bullet$};
\node [inner sep=0.8pt,outer sep=0.8pt] at (2,0) (6) {$\bullet$};
\node [inner sep=0.8pt,outer sep=0.8pt] at (3,0) (7) {$\bullet$};
\node [inner sep=0.8pt,outer sep=0.8pt] at (4,0) (8) {$\bullet$};
\node [inner sep=0.8pt,outer sep=0.8pt] at (0,-1) (2) {$\bullet$};
\draw (-2,0)--(4,0);
\draw (0,0)--(0,-1);
\draw [line width=0.5pt,line cap=round,rounded corners] (1.north west)  rectangle (1.south east);
\draw [line width=0.5pt,line cap=round,rounded corners] (2.north west)  rectangle (2.south east);
\draw [line width=0.5pt,line cap=round,rounded corners] (3.north west)  rectangle (3.south east);
\draw [line width=0.5pt,line cap=round,rounded corners] (4.north west)  rectangle (4.south east);
\draw [line width=0.5pt,line cap=round,rounded corners] (5.north west)  rectangle (5.south east);
\draw [line width=0.5pt,line cap=round,rounded corners] (6.north west)  rectangle (6.south east);
\draw [line width=0.5pt,line cap=round,rounded corners] (7.north west)  rectangle (7.south east);
\draw [line width=0.5pt,line cap=round,rounded corners] (8.north west)  rectangle (8.south east);
\end{tikzpicture}
\end{center}
\caption{The opposition diagrams of type $\mathsf{E_8}$}\label{fig:Dynkin}
\end{figure}

The \textit{fix diagram} of an automorphism of an $\sE_8$ building is given by encircling the types of the vertices of the building fixed by~$\theta$. We will use the same symbols $\sE_{8;j}$ for fixed diagrams. They are extensions of the \emph{indices} defined by Tits in \cite{Tits:66} in the sense of~\cite{MPW}. 

The opposition diagram $\sE_{8;8}$ is the opposition diagram of any non-domestic automorphism, and hence we shall not be concerned with it. Also, since the opposite geometry of a nontrivial automorphism is never empty (by \cite{AB:09}), the only automorphism with opposition diagram $\sE_{8;0}$ is the identity automorphism. Moreover, in \cite[Theorems~1 and 4]{PVMexc} we proved that each automorphism with opposition diagram $\sE_{8;1}$ is a nontrivial central collineation (and vice-versa), and that each automorphism with opposition diagram $\sE_{8;2}$ is the product of two nontrivial perpendicular root elations (and vice-versa). Thus the focus of the present paper is to classify the automorphisms with opposition diagram~$\sE_{8;4}$. 

In \cite[Theorems~5 and~6]{PVMexc} we provided examples of automorphisms with opposition diagram~$\sE_{8;4}$. These examples were certain products of $4$ perpendicular nontrivial central collineations, or certain certain homologies. In particular, all of these examples fix a chamber of the building (equivalently, they are conjugate to a member of the Borel subgroup~$B$). Moreover, in the case of $\mathsf{E}_6$, it is shown in \cite[Theorem 8]{PVMexc} that all domestic automorphisms of a thick $\sE_6$ building fix a chamber of the building. In contrast to these examples, we shall see that there exist automorphisms with opposition diagrams $\sE_{8;4}$ fixing no chamber, provided the underlying field admits certain extensions. The complete classification of such automorphisms for large buildings is given in the following theorem. We work in the non-strong parapolar space and also long root subgroup geometry $\Delta=\mathsf{E_{8,8}}(\K)$ (see Section~\ref{sec2} for the notation and terminology). We say that an automorphism of $\Delta$ is of \emph{Class} I if it pointwise fixes precisely a fully and isometrically embedded subspace isomorphic to $\mathsf{F_{4,1}}(\K,\H)$, for some quaternion division algebra $\H$ over $\K$, and \emph{of Class} II if it pointwise fixes precisely an equator geometry, while acting fixed point freely on its set of poles.

\begin{theo}\label{nofixedchamber}
Let $\theta$ be an automorphism of the building $\mathsf{E}_8(\mathbb{K})$ with $|\mathbb{K}|> 2$, and suppose that $\theta$ fixes no chamber. The following are equivalent.
\begin{compactenum}[$(1)$]
\item $\theta$ is domestic;
\item $\theta$ has opposition diagram $\mathsf{E}_{8;4}$;
\item $\theta$ is an automorphism of Class \emph{I} or Class \emph{II}. 
\end{compactenum}
Also, the building $\mathsf{E_8}(\K)$ admits an automorphism of Class \emph{I} if, and only if, $\K$ admits an associative quadratic division algebra of dimension~$4$, whereas it admits an automorphism of Class \emph{II} if, and only if, it admits a quadratic extension. 
\end{theo}

To complete the classification of domestic automorphisms of large spherical building of type $\sE_8$ we must also classify the automorphisms with opposition diagram $\mathsf{E}_{8;4}$ fixing a chamber. The techniques in this classification are more algebraic, and we give precise formulae in terms of Chevalley generators for representatives of the conjugacy classes (see Section~\ref{fixingchamber} for the notation). 

\begin{theo}\label{fixedchambers}
Let $\theta$ be an automorphism of the building $\mathsf{E}_8(\mathbb{K})$ with $|\mathbb{K}|> 2$, and suppose that $\theta$ fixes a chamber. Then $\theta$ has opposition diagram $\mathsf{E}_{8;4}$ if, and only if, $\theta$ is conjugate to one of the following elements of the standard $\sD_4$ subgroup (with $a\neq 0$ arbitrary and $c\neq 0,1$ arbitrary):
\begin{compactenum}[$(1)$]
\item $x_{0100}(a)x_{1110}(1)x_{1101}(1)x_{0111}(1)$ (a certain product of $4$ orthogonal long root elations); 
\item $x_{1111}(1)x_{0100}(a)x_{1110}(1)x_{1101}(1)x_{0111}(1)$; 
\item $h_{\varphi}(c)$ (a homology fixing a non-thick subbuilding with thick frame of type $\mathsf{E}_7$ if $c\neq -1$ and type $\mathsf{E}_7\times\mathsf{A}_1$ if $\mathsf{char}(\mathbb{K})\neq 2$ and $c=-1$);
\item $x_{\varphi}(1)h_{\varphi}(-1)$ with $\kar(\K)\neq 2$.
\end{compactenum}
\end{theo}


The unique small $\sE_8$ building $\mathsf{E}_8(2)$ (over the field with $2$ elements) is excluded from the above theorems. See Remark~\ref{rem:remain} below for a discussion.

Our results translate into group theoretic statements concerning conjugacy classes in Chevalley groups of exceptional type. To put these results into context, recall that by the \textit{Density Theorem} (see \cite[Section 22.2]{Hum:75}), if $G$ is a connected linear algebraic group over an algebraically closed field then the union of all conjugates of a Borel subgroup~$B$ is equal to~$G$. Equivalently, if $\mathcal{C}$ is a conjugacy class in $G$ then $\mathcal{C}\cap B\neq\emptyset$. This theorem is a cornerstone in the theory of algebraic groups, for example simple corollaries include the important facts that the centres of $G$ and $B$ coincide, and that the Cartan subgroups of $G$ are precisely the centralisers of maximal tori.

The statement of the Density Theorem is clearly false in the general setting of a Chevalley group $G$ over an arbitrary field, as there typically exist elements $\theta\in G$ fixing no chamber of the building $\Omega=G/B$. However our classification theorems allow us to provide analogues of the Density Theorem in this setting, showing that every conjugacy class in $G$ intersects a union of a very small number of $B$-cosets. For example, in \cite[Corollary~11]{PVMexc} we showed that in a Chevalley group of type $\sE_6$ or $\sF_4$ over a field $\FF$ we have $\mathcal{C}\cap (B\cup w_0B)\neq\emptyset$ for all conjugacy classes~$\mathcal{C}$. The existence of domestic automorphisms of $\sE_7$ and $\sE_8$ buildings fixing no chamber implies that this result does not extend to the cases $\sE_7$ and $\sE_8$, however in this paper we are able to prove the following analogue of the Density Theorem for Chevalley groups of type~$\sE_8$ (see Theorem~\ref{thm:density} for a general density theorem for all types). 

\begin{coro}\label{cor:C}
Let $G$ be a Chevalley group of type $\mathsf{E}_8$ over a field~$\mathbb{K}$ with $|\K|>2$, and let $\mathcal{C}$ be a conjugacy class in $G$. Then
\begin{compactenum}[$(1)$]
\item $\mathcal{C}\cap (B\cup w_{\mathsf{D}_4}w_0B\cup w_0B)\neq\emptyset$;
\item if $\mathbb{K}$ is quadratically closed then $\mathcal{C}\cap (B\cup w_0B)\neq\emptyset$. 
\end{compactenum}
Here $w_0$ is the longest element of~$W$, and $w_{\sD_4}$ is the longest element of the standard $\sD_4$ parabolic subgroup.
\end{coro}

In \cite{Ney-Par-Mal:23} we introduced the notion of \textit{uniclass automorphisms} of spherical buildings. These automorphisms are defined by the property that the displacement spectra $$\mathsf{disp}(\theta)=\{\delta(C,C^{\theta})\mid C\in\Omega\}$$ is contained in a single conjugacy class in the Weyl group (and it then turns out that $\mathsf{disp}(\theta)$ is necessarily equal to a conjugacy class in the Weyl group). In \cite[Theorem~1]{Ney-Par-Mal:23} we gave the classification of uniclass automorphisms, with the results for type $\mathsf{E}_8$ conditional on the classification of domestic automorphisms of buildings of type $\mathsf{E}_8$. The results of the present paper provide these results, thus completing the proof of \cite[Theorem~1]{Ney-Par-Mal:23}. In particular, the only uniclass automorphisms of buildings of type $\mathsf{E}_8$ are the identity, anisotropic automorphisms, and automorphisms of Class~I.

Theorems~\ref{nofixedchamber} and~\ref{fixedchambers} constitute the final step in the complete classification of domestic automorphisms of large spherical buildings of exceptional type with rank at least~$3$ (see \cite{PVMexc} for the classification in type $\sE_6$ and for the opposition diagrams $\sE_{7;1}$, $\sE_{7;2}$, $\sE_{8;1}$, and $\sE_{8;2}$, \cite{Lam-Mal:24} for the $\sF_4$ case, and \cite{npvv} for opposition diagrams $\sE_{7;3}$ and $\sE_{7;4}$). In the final  section of this paper we record some applications and consequences of this classification. The first application is a general version of the Density Theorem stated above (see Theorem~\ref{thm:density}). We also prove further spectral properties of Class I automorphisms of large $\sE_8$ buildings (see Section~\ref{sec:kangaroo}), and in Section~\ref{sec:biclass} we give a natural extension of the concept of uniclass automorphism, and use the classification of domestic automorphisms to classify those type preserving automorphisms of Moufang spherical buildings whose displacement spectra contains the identity and only one other conjugacy class in the Weyl group.

\begin{rem}\label{rem:remain}
As noted above, the present paper completes the classification of domestic automorphisms of large spherical buildings of exceptional type, of rank at least~$3$. We now discuss the small buildings, and the rank~$2$ case. 

The small buildings of exceptional type of rank at least $3$ are the buildings $\sE_6(2)$, $\sE_7(2)$, $\sE_8(2)$, $\sF_4(2)$, and $\sF_4(2,4)$. In these buildings the analysis is complicated by the existence of uncapped automorphisms, and indeed for these buildings additional examples of domestic automorphisms exist. For types $\sF_4$ and $\sE_6$ the classification is given in~\cite{PVMsmall} (via computation), however for the buildings $\sE_7(2)$ and $\sE_8(2)$ a different approach is required due to the enormous size of these groups, and currently the classification is unknown. 

In rank~$2$ (generalised polygons) it is not feasible to classify domestic automorphisms (because the polygons themselves do not admit a classification). However one may reasonably ask for the classification of domestic automorphisms of all Moufang polygons. For exceptional types this involves the Moufang hexagons and the Moufang octagons. The complete classification of domestic automorphisms in Moufang hexagons is given in~\cite{PVMexc2}, and it remains to classify domestic automorphisms of Moufang octagons. An approach similar to~\cite{PVMexc2} should be possible for this task, however currently the classification is unknown.

Finally, we note that for spherical buildings of classical type, partial classifications (and for some opposition diagrams complete classifications) and characterisations of domestic automorphisms are given in~\cite{PVMclass}. 
\end{rem}

\textbf{Structure of the paper.} The classification of domestic automorphisms fixing at least one chamber (Theorem~\ref{fixedchambers}) is carried out in \cref{fixingchamber}, with the preliminaries required for the proof being introduced in that same section. Essentially the approach to Theorem~\ref{fixedchambers} is via calculations in the Chevalley group using Chevalley's commutation relations. The classification of the domestic automorphisms fixing no chamber (Theorem~\ref{nofixedchamber}) is achieved using the underlying so-called \emph{long root subgroup geometry}. The necessary preliminaries on these geometries are introduced in \cref{sec2}, and the automorphisms of Class I and Class~II are described in \cref{sec3} and \cref{sec:equator} respectively. Moreover, in these sections, we characterise these automorphisms in terms of their fixed structure, proving existence and uniqueness results, and show that they are domestic with opposition diagram~$\sE_{8;4}$. In \cref{sec:class}, we show that each domestic automorphism of a large building of type $\mathsf{E_8}$ fixing no chamber is of Class I or Class II, completing the proof of Theorem~\ref{nofixedchamber}. Finally, in Section~\ref{sec:applications} we discuss applications of the classification of domestic automorphisms.

\section{Proof of Theorem~\ref{fixedchambers}}\label{fixingchamber}

In this section we classify the domestic automorphisms of a large building of type $\sE_8$ with opposition diagram~$\sE_{8;4}$ and fixing a chamber, proving Theorem~\ref{fixedchambers}. The techniques here are algebraic, making use of commutator relations in the associated Chevalley group. These techniques build on those developed in \cite{PVMexc,npvv}, and to perform the required calculations we make use of the Groups of Lie Type package in $\mathsf{MAGMA}$~\cite{MAGMA,CMT:04}. 

We first describe the algebraic setup of this section. This setup will also be used at two later occasions in the paper (Theorem~\ref{thm:equatordomestic} and Lemma~\ref{3opposite}).

Let $G_0$ be the adjoint Chevalley group of type $\sE_8$ over a field~$\K$. We adopt the sign conventions from the Groups of Lie Type package in $\mathsf{MAGMA}$ for the commutator relations. We will use the notation and conventions outlined in \cite[Section~1.1]{PVMexc}, and so in particular $\Phi$ is the root system of $G_0$, with simple roots $\alpha_1,\ldots,\alpha_8$ and $(W,S)$ is the associated Coxeter system. The fundamental coweights of $\Phi$ are denoted $\omega_1,\ldots,\omega_8$, and the coweight lattice of $\Phi$ is $P=\ZZ\omega_1+\cdots+\ZZ\omega_8$. The highest root of $\Phi$ is $\varphi=(2,3,4,6,5,4,3,2)$ (in the basis of simple roots). Let $G=G_{\Phi}(\K)$ be the subgroup of $\mathrm{Aut}(G_0)$ generated by the inner automorphisms of $G_0$ and the diagonal automorphisms, as in \cite{Hum:69,St:60}, and let $x_{\alpha}(a)$, $s_{\alpha}(t)$, and $h_{\lambda}(t)$ be the elements of $G$ described in \cite[Section~1.1]{PVMexc} (with $\alpha\in\Phi$, $\lambda\in P$, $a\in\K$, and $t\in\K^{\times}$). In particular, we have 
$
s_{\alpha}(t)=x_{\alpha}(t)x_{-\alpha}(-t^{-1})x_{\alpha}(t)
$ and $h_{\alpha}(t)=s_{\alpha}(t)s_{\alpha}(1)^{-1}$ 
for $\alpha\in\Phi$ and $t\in\K^{\times}$. We write $s_{\alpha}=s_{\alpha}(1)$ (however note that $s_{\alpha}$ is not necessarily an involution). We record the following relations for later use: $h_{\lambda}(t)h_{\mu}(t')=h_{\mu}(t')h_{\lambda}(t)$, $h_{\lambda}(t)h_{\lambda}(t')=h_{\lambda}(tt')$, and 
\begin{align}\label{eq:somerelations}
h_{\lambda}(t)x_{\alpha}(a)h_{\lambda}(t)^{-1}&=x_{\alpha}(at^{\langle\lambda,\alpha\rangle})&
s_{\alpha}h_{\lambda}(t)s_{\alpha}^{-1}&=h_{s_{\alpha}\lambda}(t)
\end{align}
for $a\in\K$ and $t\in\K^{\times}$.

For each $\alpha\in\Phi$ let $U_{\alpha}$ be the subgroup of $G$ generated by the elements $x_{\alpha}(a)$ with $a\in \K$. For $A\subseteq \Phi$ let $U_{A}$ be the subgroup generated by the groups $U_{\alpha}$, $\alpha\in A$, and let $U^+=U_{\Phi^+}$. The elements of $U^+$ and their conjugates will be called \emph{unipotent elements}. Let $H$ be the subgroup generated by the diagonal elements $h_{\lambda}(t)$ with $\lambda\in P$ and $t\in\K^{\times}$. Let $B$ be the subgroup of $G$ generated by $U^+$ and $H$. It is easy to see that an element of $B$ belongs to $U^+$ if, and only if, it fixes each or exactly one chamber of  every panel residue it stabilises (a \emph{panel} is a codimension~1 simplex).

Let $\Omega=\Omega_{\Phi}(\K)$ be the standard split spherical building associated to $G$. Thus $\Omega$ is of type $\sE_8$ with chamber set $G/B$ and Weyl distance function given by $\delta(gB,hB)=w$ if and only if $g^{-1}h\in BwB$, and by Tits' classification of spherical buildings~\cite{Tits:74} every thick building of type $\sE_8$ arises in this way for some field~$\K$, and $\Omega$ is large if and only if $|\K|>2$. 

Let $\varphi_1=\varphi_{\sE_8}$, $\varphi_2=\varphi_{\sE_7}$, $\varphi_3=\varphi_{\sD_6}$, and $\varphi_4=\alpha_7$, and let $J=\{1,6,7,8\}$ (the nodes encircled in the diagram $\sE_{8;4}$). Define
\begin{align*}
\Psi=\{\beta\in\Phi\mid \langle\alpha,\beta\rangle=0\text{ for all $\alpha\in\Phi_{\sD_4}$}\},
\end{align*}
where $\Phi_{\sD_4}$ is the parabolic subroot system of type $\sD_4$ generated by $\{\alpha_2,\alpha_3,\alpha_4,\alpha_5\}$. Let $\Psi^+=\Psi\cap \Phi^+$. Note that $\varphi_1,\varphi_2,\varphi_3,\varphi_4\in\Psi^+$.

\begin{lemma}\label{lem:magicelementE8}
The set $\Psi$ is a root system of type $\sD_4$ with positive system $\Psi^+$ and simple system $\gamma_1=\varphi_{\sE_7}$, $\gamma_2=\alpha_8$, $\gamma_3=\varphi_{\sD_6}$, and $\gamma_4=\alpha_7$. 
The element 
$$
\su=134265423456765423143546876542314354265431765876
$$
satisfies $\su^{-1}\gamma_1=\alpha_2$, $\su^{-1}\gamma_2=\alpha_4$, $\su^{-1}\gamma_3=\alpha_3$, and $\su^{-1}\gamma_4=\alpha_5$, and thus $\su^{-1}$ maps $\Psi^+$ to the positive system $\Phi_{\sD_4}^+$ of the standard $\sD_4$ parabolic subsystem.
\end{lemma}

\begin{proof}
This follows from a direct calculation.
\end{proof}

To fix conventions, we identify the simple roots $\alpha_2,\alpha_3,\alpha_4,\alpha_5$ of the $\sE_8$ root system with simple roots $\alpha_1',\alpha_3',\alpha_2',\alpha_4'$ of the $\sD_4$ system (respectively). Let $G_{\sD_4}$ be the subgroup of the $\sE_8$ Chevalley group $G$ generated by the root subgroups $U_{\alpha}$ with $\alpha\in\Phi_{\sD_4}$ (we shall use the sign conventions in $\sD_4$ inherited from the sign conventions in $\sE_8$ in $\mathsf{MAGMA}$, however since the elements below are listed up to conjugation, the particular choices turn out to be irrelevant). Let $w_{\sD_4}$ be the longest element of the $\sD_4$ Coxeter group.

%
We now prove Theorem~\ref{fixedchambers}. 

\begin{proof}[Proof of Theorem~\ref{fixedchambers}]
Suppose that $\theta$ has opposition diagram $\sE_{8;4}$. By \cite[Theorem~8.1]{npvv}, and following the first paragraph of the proof of \cite[Theorem~8.4]{npvv}, we see that $\theta$ lies in the Chevalley group $G$, and is conjugate to an element of the form
\begin{align}\label{eq:originalform}
\theta_1&=u's_{\varphi_1}^{-1}s_{\varphi_2}^{-1}s_{\varphi_3}^{-1}s_{\varphi_4}^{-1}h_{\omega_1}(c_1)h_{\omega_6}(c_2)h_{\omega_7}(c_3)h_{\omega_8}(c_4),
\end{align}
where $c_1,c_2,c_3,c_4\in\K\backslash\{0\}$, $u'\in U_{\Psi^+}$ and $\varphi_1=\varphi_{\sE_8}$, $\varphi_2=\varphi_{\sE_7}$, $\varphi_3=\varphi_{\sD_6}$, and $\varphi_4=\alpha_7$. 

We will now make explicit calculations to further restrict~$u'$. Write
\begin{align*}
u'&=x_{22343210}(a_1) x_{2 2 3 4 3 2 1 1}(a_2) x_{2 2 3 4 3 2 2 1}(a_3) x_{2 3 4 6 5 4 3 1}(a_4) x_{2 3 4 6 5 4 3 2}(a_5) x_{0 1 1 2 2 2 1 0}(a_6) x_{0 1 1 2 2 2 1 1}(a_7)\\
&\times x_{0 1 1 2 2 2 2 1}(a_8) x_{0 0 0 0 0 0 1 0}(a_9) x_{2 3 4 6 5 4 2 1}(a_{10})x_{0 0 0 0 0 0 1 1}(a_{11}) x_{0 0 0 0 0 0 0 1}(a_{12})
\end{align*}
with $a_1,\ldots,a_{12}\in\K$ (the roots appearing here are the 12 roots of $\Psi^+$). For $\alpha,\beta\in\Phi^+$ we consider the element $v\in W$ given by 
$$
x_{-\beta}(-1)Bx_{-\alpha}(-1)\theta' x_{-\alpha}(1)x_{-\beta}(1)B=BvB.
$$
Since $
v=\delta(x_{-\alpha}(1)x_{-\beta}(1)B,\theta' x_{-\alpha}(1)x_{-\beta}(1)B)$, if $\ell(v)>\ell(w_{S\backslash J}w_0)=108$ (where $J=\{1,6,7,8\}$) then we have a contradiction with the fact that $\theta$ has opposition diagram $\sE_{8;4}$ (see the proof of \cite[Theorem~8.6]{npvv}). By direct calculation (using $\mathsf{MAGMA}$) if $\alpha=(1 0 1 0 0 0 0 0)$ and $\beta=(1 2 3 4 3 2 1 0)$ with $a_6\neq c_4a_1$ then $v=s_3w_{S\backslash J}w_0$ and so $\ell(v)=\ell(w_{S\backslash J}w_0)+1$, a contradiction. Thus $a_6=c_4a_1$. Similarly, taking $\alpha=(1 0 1 0 0 0 0 0)$ and $\beta=(1 2 3 4 3 2 1 1)$ forces $a_7=c_4a_2$, taking $\alpha=(1 0 1 0 0 0 0 0)$ and $\beta=(1 2 3 4 3 2 2 1)$ forces $a_8=c_4a_3$, taking $\alpha=(0 0 0 0 1 1 0 0)$ and $\beta=(0 1 1 2 2 1 1 0)$ forces $a_9=c_3c_4a_1$, taking $\alpha=(0 0 0 0 1 1 0 0)$ and $\beta=(0 1 1 2 2 1 1 1)$ forces $a_{11}=c_3c_4a_2$, taking $\alpha=(0 0 0 0 1 1 0 0)$ and $\beta=(2 3 4 6 5 3 2 1)$ forces $a_{10}=c_3^{-1}a_3$, taking $\alpha=(0 0 0 0 1 1 1 0)$ and $\beta=(0 1 1 2 2 1 1 1)$ forces $a_{12}=-c_2c_3c_4a_3$, taking $\alpha=(0 0 0 0 1 1 1 0)$ and $\beta=(2 3 4 6 5 3 2 1)$ forces $a_4=-c_2^{-1}c_3^{-1}a_2$, and taking $\alpha=(0 0 0 0 1 1 1 1)$ and $\beta=(2 3 4 6 5 3 2 1)$ forces $a_5=c_1^{-1}c_2^{-1}c_3^{-1}a_1$.

Conjugating $\theta_1$ by the element $\su$ from Lemma~\ref{lem:magicelementE8} it follows that $\theta$ is conjugate to an element of the form $
\theta_2=uw_{\sD_4}h,
$
where $u,h\in G_{\sD_4}$ are of the form
\begin{align*}
u&=x_{1000}(a_1) x_{1100}(a_2) x_{1101}(-a_3) x_{1111}(c_2^{-1}c_3^{-1}a_2) x_{1211}(-c_1^{-1}c_2^{-1}c_3^{-1}a_1)x_{0010}(c_4a_1)\\
&\times x_{0110}(c_4a_2) x_{0111}(-c_4a_3) x_{0001}(c_3c_4a_1)x_{1110}(c_3^{-1}a_3)x_{0101}(-c_3c_4a_2)x_{0100}(-c_2c_3c_4a_3)\\
h&=h_{1000}(-c_1c_2^2c_3^3c_4^2)h_{0100}(c_1^2c_2^3c_3^4c_4^2)h_{0010}(-c_1c_2^2c_3^3c_4)h_{0001}(-c_1c_2^2c_3^2c_4)
\end{align*} 

Considering $h^{-1}\theta_2h$, and setting $a=t_3t_4a_1$, $b=t_3t_4a_2$, and $c=t_2t_3t_4a_3$, we see that $\theta$ is conjugate to an element of the form $uw_{\sD_4}h$ with
\begin{align*}
u&=x_{1000}(t_2t_3t_4a)x_{1100}(-t_1t_2t_3t_4b)x_{1101}(t_1t_2t_3t_4c)x_{1111}(-t_1t_2^2t_3^2t_4b)x_{1211}(t_1t_2^2t_3^2t_4a)x_{0010}(t_2t_3a)\\
&\qquad x'_{0110}(t_1t_2t_3b)x'_{0111}(t_1t_2t_3c)x'_{0001}(t_2a)x'_{1110}(t_1t_2t_3^2t_4c)x'_{0101}(t_1t_2b)x'_{0100}(t_1c)\\
h&=h_{1000}(t_1t_2^2t_3^3t_4^2)h_{0100}(t_1^2t_2^3t_3^4t_4^2)h_{0010}(t_1t_2^2t_3^3t_4)h_{0001}(t_1t_2^2t_3^2t_4).
\end{align*}
This element of $G_{\sD_4}$ appears in \cite[Theorem~8.6]{npvv} (in the $\sD_4$ subgroup of the $\sE_7$ Chevalley group, with superficial differences in signs reflecting the different sign choices in $\sE_7$ and $\sE_8$). The $\sD_4$ analysis in \cite[Theorem~8.10]{npvv} applies and it follows that $\theta$ is conjugate to an element of the form listed in the statement of the theorem. 

It remains to show that each element $\theta$ of the form (1), (2), (3), or (4) in the statement of the theorem is domestic with opposition diagram~$\sE_{8;4}$. Reversing the conjugations performed in \cite[Theorem~8.10]{npvv} shows that $\theta$ is conjugate to an element of the form~(\ref{eq:originalform}), and such an element maps the base chamber to Weyl distance $s_{\varphi_1}s_{\varphi_2}s_{\varphi_3}s_{\varphi_4}=w_{\sD_4}w_0$. This shows that $\theta$ maps some type $\{1,5,6,7\}$ simplex to an opposite, and hence $\theta$ is either domestic with diagram~$\sE_{8;4}$, or $\theta$ is not domestic (by the classification of possible opposition diagrams). Thus it is sufficient to prove that elements $\theta$ of the form (1), (2), (3), or (4) are domestic. 

Consider $\theta=x_{0100}(a)x_{1110}(1)x_{1101}(1)x_{0111}(1)$ with $a\neq 0$. By \cite[Lemma~8.9]{npvv} this element is conjugate to an element of the form $\theta'=x_{1211}(\pm a)x_{1000}(1)x_{0010}(1)x_{0001}(1)$, and then with $\su$ as in Lemma~\ref{lem:magicelementE8} we have $\su\theta'\su^{-1}=x_{\varphi_1}(\pm a)x_{\varphi_2}(1)x_{\varphi_3}(1)x_{\varphi_4}(1)$. This element is domestic with opposition diagram~$\sE_{8;4}$ by \cite[Theorem~3.1]{PVMexc}. 

Consider $\theta=x_{1111}(1)x_{0100}(a)x_{1110}(1)x_{1101}(1)x_{0111}(1)$ with $a\neq 0$. Then
$$
\su\theta\su^{-1}=x_{23465431}(\pm 1)x_{00000001}(\pm a)x_{23465421}(\pm 1)x_{22343221}(\pm 1)x_{01122221}(\pm 1),
$$
and by \cite[Lemma~3.4(3)(b)]{PVMexc} this element is domestic (and hence has diagram~$\sE_{8;4}$).

The element $h_{\varphi}(c)$ with $c\neq 0,1$ is domestic (with diagram $\sE_{8;4}$) by \cite[Theorem~4.7]{PVMexc} (note that $\varphi=\omega_8$). Finally, the element $x_{\varphi}(1)h_{\varphi}(-1)$ with $\kar(\K)\neq 2$ is domestic with diagram $\sE_{8;4}$ by \cite[Proposition~8.11]{npvv}, completing the proof. 
\end{proof}

\section{Geometric preliminaries}\label{sec2}

The remainder of this paper is focused on proving Theorem~\ref{nofixedchamber}, and the techniques used are primarily geometric. Indeed the classification and description of the domestic collineations that do not fix any chamber makes use of the so-called \emph{long root subgroup geometry} of type $\mathsf{E_8}$, which is a \emph{parapolar space}. We will use the language of the theory of parapolar spaces, and we provide a brief introduction below. 

Note that we, in \cref{nofixedchamber}, not only provide a necessary and sufficient condition for the existence of domestic collineations fixing no chamber in terms of fixed substructures, but also give a necessary and sufficient algebraic condition on the field of definition for such substructures to exist. All domestic collineation corresponding to a given fix structure isomorphic to a metasymplectic space  form a group, which we describe abstractly below.
\subsection{Quaternion algebras}\label{quatalg}
A building of type $\mathsf{E_8}$ is completely determined by a given field $
\K$ (and then each residue of type $\mathsf{A_2}$ conforms to a projective plane naturally corresponding to a $\K$-vector space of dimension 3). 

Let $\L/\K$ be a quadratic extension (not necessarily separable) and let $\L\to\L:x\mapsto\overline{x}$ be the unique automorphism of $\L$ pointwise fixing $\K$ and such that $\Norm(x):=x\overline{x}\in\K$, for all $x\in \L$. Note that this map, which we call the \emph{standard involution} is trivial in the inseparable case, and is really an involution in the separable case. An element $\Norm(x)$, $x\in\L$, is called a \emph{norm}. Suppose now there exists $a\in\K$  such that $a$ is not a norm. Then we can define a division algebra $\H$ over $
K$ (via the so-called Cayley--Dickson process), consisting of all pairs $(x,y)\in\L\times\L\cong\K\times\K\times\K\times\K$ with multiplication
\[(x,y)\cdot(x',y')=(xx'+a\overline{y}y',\overline{x}y'+x'y), \mbox{ for all }x,x',y,y'\in\K.\]
If $\L/\K$ is separable, then $\H$ is a proper skew field, called a \emph{quaternion division algebra}. If $\L/\K$ is inseparable (this can only happen if the characteristic of $\K$ is $2$), then $\H$ is just a field extension of $\K$ of degree $4$, which, for ease of formulation, we will also call a \emph{quaternion division algebra over $\K$}, sometimes preceded by the adjective \emph{inseparable}.  

Now let $\H$ be a quaternion division algebra over $\K$, defined via $\L/\K$ as above. We embed $\L$ in $\H$ by taking the first component of each pair. We embed $\K$ in $\H$ via the natural embedding of $\K$ in $\L$. We can extend the standard involution as follows. For $z=(x,y)\in\H$ we set $\overline{z}=(\overline{x},-y)$. Then $\Norm(z):=z\cdot\overline{z}=\overline{z}\cdot z\in\K$. In the inseparable case we again have $\Norm(z)=z^2$. It follows that, in that case, an element has norm $1$ if and only if it is itself equal to $1$. However, in the separable case there are many elements with norm $1$ and these form a (non-commutative) subgroup of the multiplicative group of $\H$, which we denote by $G(\K,\H)$. If $\H$ is inseparable, then $G(\K,\H)$ denotes the additive group (of exponent $2$) of elements $(x,y)\in\H$ such that $x=x^2+ay^2$ (with above notation). 

Let $\H$ be a quaternion division algebra over $\K$. Then the map
\[\rho\colon \H\times\H\times\H\to\K\times\K\times\K\times\H\times\H\times\H\colon(x,y,z)\mapsto(x\overline{x},y\overline{y},z\overline{z},\overline{y}z,\overline{z}x,\overline{x}y)\] is called the \emph{Veronesean map}. Its image in the associated projective space $\PG(14,\K)$ is called the \emph{quaternion Veronese variety}.  If we consider the triple $(x,y,z)$ as homogeneous coordinates of a point of the projective plane $\PG(2,\H)$ over $\H$, then the image of a line of $\PG(2,\H)$ through the Veronese map represents a quadric of Witt index $1$ in a subspace of dimension $5$ with standard equation $X_{-1}X_1=\Norm(X_0)$, where $X_{-1},X_1\in\K$ and $X_0\in\H$, the latter viewed as $4$-dimensional vector space over $\K$.  

A quaternion division algebra over $\K$ is a special case of a \emph{quadratic alternative} division algebra over $\K$, that is, an algebra over $\K$ satisfying the alternative laws $a(ab)=a^2b$ and $(ab)b=ab^2$ and such that every element $x$ satisfies a quadratic equation $x^2-t(x)x+n(x)=0$, where $t$ is a linear form and $n$ a quadratic form. In our case, we have $n\equiv\Norm$ and $t(x)=x+\overline{x}$. Also a quadratic extension of $\K$ is in a natural way a quadratic algebra over $\K$. The quadratic alternative division algebras are classified, but we will not need this result (however, see for instance Chapter 10 of \cite{Tits:74}). One of the examples is the split Cayley algebra $\mathbb{O}'$ over $\K$, a non-associative $8$-dimensional algebra over $\K$. One can define a Veronese variety over $\mathbb{O}'$ (as in \cite{Sch-Sch-Mal-Vic:23}). If $\mathbb{O}'$ contains a subalgebra $\H$ isomorphic to a quaternion algebra, then one can restrict the coordinates from $\mathbb{O}'$ down to $\H$ in the definition of this variety to obtain a \emph{standard embedding} of the quaternion Veronese variety in the Veronese variety over $\mathbb{O}'$.

\subsection{Parapolar spaces}
We assume the reader to be familiar with the basic theory of buildings from the simplicial complex point of view as explained in the first three chapters of \cite{Tits:74}. By \cite[Theorem~6.13]{Tits:74}, we know that, up to isomorphism, there exists exactly one building of type $\mathsf{E}_n$, $n=6,7,8$, with ground field $\K$, that is, such that all irreducible residues of rank $2$ are projective planes over $\K$. We shall denote that building by $\mathsf{E}_n(\K)$. By \cite[Theorem~10.2]{Tits:74}, a building of type $\mathsf{F_4}$ is uniquely determined by a field $\K$ and a quadratic alternative division algebra $\A$. We denote such a building by $\mathsf{F_4}(\K,\A)$. By convention, the types $1$ and $2$ are such that the residues of simplices of cotype $\{1,2\}$ correspond to projective planes over $\K$, see \cref{fig2}. 

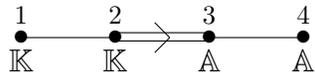
\begin{figure}[h]
\begin{center}
\begin{tikzpicture}[scale=0.5]
\node at (0,0.3) {};
\node [inner sep=0.8pt,outer sep=0.8pt] at (-3.75,0) (4) {$\bullet$};
\node [inner sep=0.8pt,outer sep=0.8pt] at (-3.75,0.6) (4) {\footnotesize 1};
\node [inner sep=0.8pt,outer sep=0.8pt] at (-3.75,-0.6) (4) {$\mathbb{K}$};
\node [inner sep=0.8pt,outer sep=0.8pt] at (-1.25,0) (3) {$\bullet$};
\node [inner sep=0.8pt,outer sep=0.8pt] at (-1.25,0.6) (3) {\footnotesize 2};
\node [inner sep=0.8pt,outer sep=0.8pt] at (-1.25,-0.6) (3) {$\mathbb{K}$};
\node [inner sep=0.8pt,outer sep=0.8pt] at (1.25,0) (3) {$\bullet$};
\node [inner sep=0.8pt,outer sep=0.8pt] at (1.25,0.6) (2) {\footnotesize 3};
\node [inner sep=0.8pt,outer sep=0.8pt] at (1.25,-0.6) (2) {$\mathbb{A}$};
\node [inner sep=0.8pt,outer sep=0.8pt] at (3.75,0) (2) {$\bullet$};
\node [inner sep=0.8pt,outer sep=0.8pt] at (3.75,0.6) (1) {\footnotesize 4};
\node [inner sep=0.8pt,outer sep=0.8pt] at (3.75,-0.6) (1) {$\mathbb{A}$};
\draw (-3.75,0)--(-1.25,0);
\draw (-1.25,0.13)--(1.25,0.13);
\draw (-1.25,-0.09)--(1.25,-0.09);
\draw (1.25,0)--(3.75,0);
\draw (.2,0)--(-.2,0.4);\draw (.2,0)--(-.2,-0.4); 
\end{tikzpicture}\hspace{2cm}
\end{center}\caption{The Dynkin diagram of type $\mathsf{F_4}$ with Bourbaki labelling\label{fig2}}
\end{figure}

Likewise, a building of type $\mathsf{A}_n$, $n\geq 3$, or $\mathsf{D}_n$, $n\geq 4$, is, up to an automorphism of the Coxeter diagram, completely determined by a skew field $\L$ and a field $\K$, respectively, and we denote the building as $\mathsf{A}_n(\L)$ and $\mathsf{D}_n(\K)$, respectively.

We now briefly explain how to define point-line geometries from spherical buildings. A \emph{point-line geometry} $\Gamma$ is a pair $(X,\cL)$, where $X$ is the \emph{point set} and $\cL$ is the \emph{line set}; for our purposes it suffices that each line is a subset of the point set. We us the Greek letters $\Gamma$ and $\Delta$, possibly furnished with subscripts, for geometries. A point-line geometry is called a \emph{partial linear space} if any pair of distinct points $x,y$ is contained in at most one line, which is then denoted by $xy$. It is called \emph{thick} if each line contains at least three points. Two (not necessarily distinct) points $x,y$ contained in a common line are said to be \emph{collinear}, in symbols $x\perp y$. The set of points collinear to a given point $x$ is denoted by $x^\perp$. A permutation $\theta$ of the point set such that triples of collinear points are mapped to triples of collinear points by both $\theta$ and $\theta^{-1}$, is called a \emph{collineation}. 

Let $\Omega$ be a (simplicial) building over some type set $S$, and pick $j\in S$. There is an associated geometry $\Delta$ with point set the set of vertices of $\Omega$ of type $j$. The lines are the sets of vertices completing the simplices of type $S\setminus\{j\}$ to a chamber (several simplices of type $S\setminus\{j\}$ can give rise to the same line). If $\Omega$ is thick, then so is $\Delta$. If $\Omega$ is irreducible and spherical, say of type $\mathsf{X}_n$, then we say that $\Delta$ has type $\mathsf{X}_{n,j}$. Using Bourbaki labelling \cite{Bourbaki}, the most popular geometries are those of type $\mathsf{A}_{n,1}$ (projective spaces) and $\mathsf{B}_{n,1},\mathsf{D}_{n,1}$ (polar spaces, see below). In the exceptional cases we denote the point-line geometry of type $\mathsf{E}_{n,j}$ over the field $\K$ by $\mathsf{E}_{n,j}(\K)$. For the building $\mathsf{F_4}(\K,\A)$, we denote the geometry of type $\mathsf{F}_{4,j}$ by $\mathsf{F}_{4,j}(\K,\A)$.  Also, geometries of type $\mathsf{A}_{n,j}$ and $\mathsf{D}_{n,j}$ corresponding to the buildings $\mathsf{A}_n(\K)$ and $\mathsf{D}_{n}(\K)$ are denoted as $\mathsf{A}_{n,j}(\K)$ and $\mathsf{D}_{n,j}(\K)$, respectively. 

Geometries arising from thick spherical buildings as explained in the previous paragraph are called \emph{Lie incidence geometries} (see \cite{BCN} for further details).

We now recall some terminology from the theory of parapolar spaces (see for instance \cite[Chapter~16]{Shu:11} for more background). 

Let $\Gamma=(X,\cL)$ be a point-line geometry. A \emph{subspace} is a set of points which contains each line with which it shares at least $2$ distinct points. A \emph{hyperplane} is a subspace intersecting each line in at least one point, called \emph{proper} when it does not coincide with $X$ itself.  A \emph{singular} subspace is one for which every pair of points is collinear. The \emph{point graph} of $\Gamma$ is the graph with vertices the members of $X$, with distinct vertices adjacent if they are collinear. A subspace is called \emph{convex} if the corresponding set of vertices in the point graph is a convex set (that is, closed under taking shortest paths between its vertices).  This notion should not be confused with the notion of convexity of subcomplexes of buildings, viewed as a simplicial complex. We will also need the latter notion, and we shall always make it very clear which kind of convexity is meant. Convexity in the sense of simplicial complexes and buildings means that a set is closed under taking projections in the building-theoretic sense (and we follow here Abramenko \& Brown \cite[Section~4.11.2, Definition~4.120]{Abr-Bro:08}, and not Tits \cite[Section~1.5]{Tits:74}). 

A \emph{polar space} is a point-line geometry $\Gamma=(X,\cL)$ in which $x^\perp$ is a proper hyperplane for all $x\in X$ (then $\Gamma$ is automatically a partial linear space, see~\cite{Bue-Shu:74}). We will always assume that a polar space is thick, unless explicitly stated otherwise, in which case we talk about \emph{weak} polar spaces. Polar spaces in this paper will always have some finite rank $r$, that is,  each maximal singular subspace is a projective space with finite dimension equal to $r-1$.

\begin{definition}
A \emph{parapolar space} is a point-line geometry with connected point graph such that 
(1) and two points at distance $2$ in the point graph either they have a unique common neighbour in that graph, or are contained in a convex subgeometry isomorphic to a polar space (called a \emph{symplecton}, or \emph{symp} for short), and (2) every line is contained in a symp. 
A parapolar space is called \emph{strong} if the first possibility in (1) never occurs. It is called \emph{proper} if it is not a polar space itself. A convex subspace which, when endowed with the lines contained in it is itself a proper parapolar space, is called a \emph{para}.  
\end{definition}

If $p,q$ are points of a parapolar space we write $p\JJoin q$ to indicate that $p,q$ are at distance~$2$ in the point graph, and have exactly one common neighbour, and we write $p^{\Join}$ for the set of all points $q'$ with $p\JJoin q'$. If $p\JJoin q$ then we say that that $p$ and $q$ are \textit{special}, and we denote the common collinear point by $\mathfrak{c}(p,q)$. If $p$ and $q$ are not collinear but are contained in a symp then we write $p\pperp q$, and we denote the (unique) symp containing $p$ and $q$ by $\xi(p,q)$, and we call $p$ and $q$ \emph{symplectic}. Two symps of the same rank intersecting in a maximal singular subspace of both are called \emph{adjacent}. If all symps of a parapolar space have the same rank $r$, then the polar space is said to have \emph{rank $r$}.

Local properties of parapolar spaces, that is, properties involving subspaces and symps that contain a common point, can best be seen and proved in the corresponding \emph{point residual geometry}, which we define now. Let $x$ be a point of some parapolar space $\Delta$ of rank at least $3$. Then $\Res_\Delta(x)$ is the point-line geometry with point set the set of lines of $\Delta$ containing $x$, and the lines of $\Res_\Delta(x)$ are the planar line pencils with vertex $x$, that is, the sets of lines through $x$ in a given plane (a singular subspace of dimension~$2$) containing~$x$. Point residuals of Lie incidence geometries are also Lie incidence geometries, and they correspond to the star $\mathsf{St}(x)$ defined for buildings in \cite[\S1.1]{Tits:74}. 

It turns out that most Lie incidence geometries are parapolar spaces. We now introduce in some more detail the ones we will need in the present paper. 

We already used the notation $\PG(n,\A)$ for the projective space of dimension $n$ over an associative division algebra $\A$, and this is nothing else than the Lie incidence geometry $\mathsf{A}_{n,1}(\A)$. Also, $\mathsf{A}_{n,n}(\A)$ is isomorphic to $\PG(n,\B)$, where $\B$ is the opposite division ring of~$\A$. 

The Lie incidence geometries $\mathsf{D}_{n,1}(\K)$, for $\K$ a field, are polar spaces, said to be of \emph{hyperbolic type}. These (thick) polar spaces are \emph{top-thin}, meaning that every singular subspace of dimension $n-2$ is contained in precisely two singular subspaces of dimension $n-1$ (these are the maximal singular subspaces). Note that singular subspaces are projective spaces and so we can indeed speak about their dimension. A polar space $\mathsf{D}_{n,1}(\K)$ of hyperbolic type can be identified with a hyperbolic quadric in $\PG(2n-1,\K)$ with standard equation (using standard notation) $X_{-n}X_n+X_{-n+1}X_{n-1}+\cdots+X_{-1}X_1=0$. They have rank  $n$, as defined above (it is also the Witt index of the quadric).

The parapolar spaces $\mathsf{F_{4,1}}(\K,\A)$ and $\mathsf{F_{4,4}}(\K,\A)$ are sometimes called \emph{metasymplectic spaces}, and we shall also use this name.   We add the adjective \emph{quaternion} if $\A$ is quaternion over $\K$. The stars of vertices of type $1$ in $\mathsf{F_4}(\K,\A)$ are buildings corresponding to polar spaces and denoted $\mathsf{C_3}(\A,\K)$, in conformity with the direction of the arrow in the related Dynkin diagrams, and hence the point residuals of $\mathsf{F_{4,1}}(\K,\A)$ are Lie incidence geometries $\mathsf{C_{3,3}}(\A,\K)$. Dually, the stars of vertices of type $4$ in $\mathsf{F_4}(\K,\A)$ are buildings, also corresponding to polar spaces, and denoted $\mathsf{B_3}(\K,\A)$.  Hence the point residuals of $\mathsf{F_{4,4}}(\K,\A)$ are Lie incidence geometries $\mathsf{B_{3,3}}(\K,\A)$. It follows that the symps of $\mathsf{F_{4,1}}(\K,\A)$ are polar spaces isomorphic to $\mathsf{B_{3,1}}(\K,\A)$, and the symps of $\mathsf{F_{4,4}}(\K,\A)$ are polar spaces isomorphic to $\mathsf{C_{3,1}}(\A,\K)$. The geometries  $\mathsf{B_{3,3}}(\K,\A)$ and  $\mathsf{C_{3,3}}(\A,\K)$ are called \emph{dual polar spaces (of rank $3$)}. If $\A$ is a quadratic extension of $\K$, we more exactly call it a \emph{quadratic dual polar space}; if $\A$ has dimension $4$ over $\K$, then we call it a \emph{quaternion dual polar space}. 

The parapolar space $\mathsf{E_{6,1}}(\K)$ is strong and has diameter $2$ (that is, its point graph has diameter $2$). The symps are all isomorphic to $\mathsf{D_{5,1}}(\K)$, and the point residuals to $\mathsf{D_{5,5}}(\K)$. The parapolar space $\mathsf{E_{7,7}}(\K)$ is strong and has diameter $3$. The symps are all isomorphic to $\mathsf{D_{6,1}}(\K)$ and the point residuals to $\mathsf{E_{6,1}}(\K)$. The geometries $\mathsf{E_{6,1}}(\K)$ and $\mathsf{E_{7,7}}(\K)$ are the so-called \emph{minuscule geometries} for $\mathsf{E_{6}}(\K)$ and $\mathsf{E_{7}}(\K)$, respectively. They contain maximal singular subspaces (of dimension $5$ and $6$, respectively) that are not contained in any symp. Also, no hyperplane of such maximal singular subspaces is contained in any symp, and we refer to these hyperplanes as $4'$-spaces and $5'$-spaces, respectively.  

We also note that the Veronese variety over the split Cayley algebra $\mathbb{O}'$ over $\K$ defines an embedding, called the \emph{standard inclusion}, of $\mathsf{E_{6,1}}(\K)$ in $\PG(27,\K)$. 

The parapolar spaces $\mathsf{E_{6,2}}(\K)$, $\mathsf{E_{7,1}}(\K)$, $\mathsf{E_{8,8}}(\K)$ and $\mathsf{F_{4,1}}(\K,\A)$ are the so-called \emph{long root subgroup geometries} of exceptional type. They are non-strong parapolar spaces of diameter $3$. 
 
In the next subsection we collect some facts which hold for both metasymplectic spaces and the long root subgroup geometries mentioned above, noting that these structures are  \emph{hexagonic geometries} as introduced by Kasikova \& Shult \cite{Kas-Shu:02}; the name was coined by Shult in \cite[\S13.7]{Shu:11}. These geometries are all parapolar spaces of diameter 3 which are not strong. Points at mutual distance $3$ represent opposite vertices in the corresponding building, hence they shall be called \emph{opposite points}. If a subspace $\Gamma'$ of such a geometry $\Gamma$ is (when endowed with the lines of $\Gamma$ completely contained in $\Gamma'$) a non-strong parapolar space of diameter $3$ itself, then we say that $\Gamma'$ \emph{is isometrically embedded in $\Gamma$} if any pair of points of $\Gamma'$ is symplectic, special or opposite in $\Gamma'$ if, and only if, it is symplectic, special or opposite, respectively, in $\Gamma$. 

When a subgeometry $\Gamma'$ of a Lie incidence geometry $\Gamma$ is also a Lie incidence geometry (for instance when $\Gamma'$ is a symp) with a different underlying opposition relation, then we use the notion of \emph{$\Gamma'$-opposition} to denote elements of $\Gamma'$ which are opposite in $\Gamma'$, but not necessarily in $\Gamma$. 

We already mentioned that an automorphism of a spherical building mapping no chamber to an opposite is called domestic. More specifically, if $J$ is a set of types of a spherical building, then an automorphism of it is called \emph{$J$-domestic} if it does not map any simplex of type $J$ to an opposite. Moreover, given an automorphism $\theta$ of a spherical building, a simplex is called \emph{domestic with respect to $\theta$} if it is not mapped onto an opposite (or simply \textit{domestic} if $\theta$ is implicit).

\subsection{Some properties of long root subgroup geometries and metasymplectic spaces}
In this subsection, $\Gamma=(X,\cL)$ is a geometry isomorphic to either some metasymplectic space, or to one of the long root subgroup geometries $\mathsf{E_{6,2}}(\K)$, $\mathsf{E_{7,1}}(\K)$ or $\mathsf{E_{8,8}}(\K)$, with $\K$ an arbitrary field. We collect some facts in the lemmas below. 

We already discussed the point residuals of metasymplectic spaces. Now we consider those of the other exceptional hexagonic geometries (c.f. \cite[\S13.7]{Shu:11}) of rank at least~$3$.

\begin{lemma}\label{stars}
The symps and point residuals of $\mathsf{E_{6,2}}(\K)$, $\mathsf{E_{7,1}}(\K)$ and $\mathsf{E_{8,8}}(\K)$ are Lie incidence geometries isomorphic to $\mathsf{D_{4,1}}(\K)$ and $\mathsf{A_{5,3}}(\K)$, $\mathsf{D_{5,1}}(\K)$ and $\mathsf{D_{6,6}}(\K)$, and $\mathsf{D_{7,1}}(\K)$ and $\mathsf{E_{7,7}}(\K)$, respectively.  
\end{lemma}

\begin{proof}
This can be read off the corresponding Coxeter diagram $D$ noting that the type of the residue of a point of type $t$ is given by the Coxeter diagram obtained from $D$ by removing the vertex corresponding to type $t$, and taking as new point set the type adjacent to $t$. The diagrams of the symps in $\mathsf{X}_{n,t}(\mathbb{K})$ (with $\mathsf{X}_{n,t}\in\{\sE_{6,2},\sE_{7,1},\sE_{8,8}\}$) are obtained by deleting the (always unique) set of vertices from $D$ such that the node corresponding to type $t$ represents the point set of a polar space in the modified diagram. For instance, the node labelled $1$ in the diagram $\mathsf{E_7}$ only becomes the node corresponding to the point set of a polar space if one deletes nodes $6$ and $7$ from the $\sE_7$ diagram, and then a $\mathsf{D_5}$ diagram remains. 
\end{proof}

\begin{lemma}\label{KasShu}
Let $p\in X$ be a point and $\pi$ a plane of $\Gamma$ containing some line $L\in\cL$. If $p$ is special to some point $x\in L$, and $\mathfrak{c}(p,x)\in\pi$, then there is a unique point $y\in L$ symplectic to $p$; all other points of $L$ are special to $p$.
\end{lemma}

\begin{proof}
This is a rephrasing of axiom (H2) of hexagonic geometries in \cite[\S13.7]{Shu:11}.
\end{proof}

\begin{lemma}[{\cite[Lemma 2(v)]{Coh-Iva:07}}]\label{spspop}
If $a\perp b\perp c\perp d$ is a path in $\Gamma$ with $a\JJoin c$ and $b\JJoin d$, then $a$ is opposite~$d$.
\end{lemma}


\begin{lemma}\label{perp-perp}
Let $p$ and $q$ be opposite points. Set $X_p=p^\perp\cap q^{\Join}$ and $X_q=q^\perp\cap p^{\Join}$. Then both $X_p$ and $X_q$, furnished with the lines contained in it, are geometries isomorphic to $\Res_\Delta(p)$. Also,  $\rho:X_p\to X_q:x\mapsto x^\perp\cap X_q$ is an isomorphism of Lie incidence geometries. Finally, we have the correspondences for all $x,y\in X_p$:
\[\begin{array}{rcl}
x=y & \iff & x\perp \rho(y),\\
x\perp y & \iff & x\pperp \rho(y),\\
x\pperp y & \iff & x\Join \rho(y),\\
x\Join y & \iff & x\mbox{ is opposite }\rho(y).
\end{array}\] 
\end{lemma}

\begin{proof}
This follows from Lemma~\ref{spspop} and \cite[Proposition~3.29]{Tits:74}.
\end{proof}

\begin{lemma}[{\cite[Lemma~4.10]{Sch-Sch-Mal:23}}] \label{lemadd}
If a point $p$ is special to all points of a line $L$, then there exists a line $M$ all points of which are collinear to $p$ and $M$ is opposite $L$ in a symp $\xi$. The map $L\to M:x\mapsto \mathfrak{c}(p,x)$ is a bijection.
\end{lemma}

Combining \cref{stars} and \cref{lemadd} gives

\begin{lemma}\label{pointline}
Let $p$ be a point collinear to at least one point of a symp to which it does not belong. Then $p$ is collinear to at least one line of the symp.
\end{lemma}

The next result follows from considering an apartment, as given in \cite{HVM-MV}.

\begin{lemma}\label{pointsymp}
Let $p$ be a point of $\Gamma$ and $\xi$ a symp containing a point opposite $p$. Then there is a unique point $x\in\xi$ symplectic to $p$. All points of $\xi$ collinear to but distinct from $x$ are special to $p$ and all points of $\xi$ not collinear to $x$ are opposite $p$.
\end{lemma}

\begin{proof}
This follows by choosing an apartment containing $p$ and $\xi$, and using the explicit descriptions of apartments given in \cite[Section~7]{HVM-MV}. 
\end{proof}

In our geometric context, the \emph{projection} of an element $x$ onto another element $y$ is the element $z$ where $z\neq y$ and $\{y,z\}$ is the building theoretic projection of $x$ onto $y$ (see \cite[\S 3.19]{Tits:74}). 

\begin{lemma}\label{projline}
The projection of a plane $\pi$ onto a line $L$ opposite some line $M\subseteq \pi$ is the unique symp $\xi$ containing $L$ and all centres of the special pairs $\{p,x\}$, with $x\in L$ and $p$ the unique point of $\pi$ special to al points of $L$. Conversely, the projection of a symp $\xi$ onto a line $L$ opposite some line $M\subseteq \xi$ is the plane spanned by $L$ and the unique point $p$ collinear to all points of $\xi$ that are symplectic to some point of $L$.
\end{lemma}

\begin{proof}
This follows from \cref{spspop}. 
\end{proof}

We will also need the following property. 

\begin{lemma}\label{projplane'}
The projection of a symp $\xi$ onto a plane $\pi$ opposite some plane $\alpha\subseteq \xi$ is the unique symp $\zeta$ containing $\pi$ and such that some plane of $\zeta$ is contained in a symp together with some plane of $\xi$.  
\end{lemma}

\begin{proof}
For metasymplectic spaces, this is obvious since $\zeta$ is the unique symp through $\pi$ that is not opposite $\xi$.

Now we argue in $\mathsf{E_{6,2}}(\K)$, $\mathsf{E_{7,1}}(\K)$ and $\mathsf{E_{8,8}}(\K)$. Since $\alpha$ contains points opposite any point of $\pi$, \cref{pointsymp} implies that each point $x$ of $\pi$ is symplectic to a unique point $x^\theta$ of $\xi$ and the set of those is a plane $\alpha'\subseteq\xi$. Now  let $L$ be a line of $\pi$ and select two distinct points $x_1,x_2\in L$. Then \cref{pointline} implies that $x_2$ is collinear to some line $L_1\subseteq\xi(x_1,x_1^\theta)$, and the point $c=L_1\cap (x_1^\theta)^\perp$ is the centre of the special pair $\{x_2,x_1^\theta\}$ (this is indeed a special pair and not a symplectic one as $x_2^\theta\neq x_1^\theta$ is the unique point of $\xi$ symplectic to $x_2$, and $x_1^\theta\perp x_2^\theta$).   Interchanging the roles of $L$ and $L^\theta$, we find that $c$ is also contained in $\xi(x_2,x_2^\theta)$, implying that $c$ is contained in $\xi(x_1,x_1^\theta)\cap\xi(x_2,x_2^\theta)$, and hence also in each $\xi(x,x^\theta)$, $x\in L$. Note that this construction of $c$ implies that $c\perp L$, and symmetrically, one also has $c\perp L^\theta$.  If $c$ were collinear to $\pi$, then \cref{KasShu} would imply that each point on $L^\theta$ is symplectic to at least two points of $\pi$, a contradiction. Hence $c$ and $\pi$ are contained in a unique symp $\zeta$. Now $x_1^\theta$ is collinear to a line $K_1$ in $\zeta$, and obviously $K_1$ contains the centres of the special pairs $\{x,x_1^\theta\}$, for all $x\in\pi$.  Similarly there is line $K_2$ in $\zeta$ through $c$ carrying all  the centres of the special pairs $\{x,x_2^\theta\}$, $x\in\pi$. Varying $L$ in $\pi$, we see that the centres of the special pairs $\{x,x'\}$, $x\in\pi$, $x'\in \alpha'$, $x'\neq x^\theta$, are contained in and fill the plane $\pi'$ spanned by $K_1$ and $K_2$. Moreover each point of $\alpha'$ is collinear to a line of $\pi'$ and vice versa, and this yields a unique symp $\xi'$ through $\pi'$ and $\alpha'$. 

Considering the corresponding building, it is clear that every apartment containing $\pi$ and $\xi$ contains $\alpha'$, and hence $\pi'$ and hence $\zeta$. This shows that $\zeta$ is the projection of $\xi$ onto $\pi$.     
\end{proof}

We conclude this subsection with two properties of the geometry $\mathsf{E_{7,7}}(\K)$. 

\begin{lemma}\label{pointsympE77}
Let $p$ be a point of $\mathsf{E_{7,7}}(\K)$ and $\xi$ a symp containing a point opposite $p$. Then there is a unique point $x\in\xi$ collinear to $p$. All points of $\xi$ collinear to but distinct from $x$ are symplectic to $p$ and all points of $\xi$ not collinear to $x$ are opposite $p$.
\end{lemma}

\begin{proof}
This follows in the same way as Lemma~\ref{pointsymp}, using the description of apartments of type $\sE_7$ in \cite[Section~7.2]{HVM-MV}. 
\end{proof}

\begin{lemma}\label{lemmaE78}
Let $\xi$ be a symp of $\mathsf{E_{7,7}}(\K)$, and let $Y$ be the set of points collinear to a unique point of $\xi$. Then every symp entirely contained in the complement of $Y$ is either adjacent to $\xi$ or coincides with it. 
\end{lemma}

\begin{proof}
Each point of an opposite symp (to $\xi$) is collinear to a unique point of $\xi$ (c.f. Lemma~\ref{pointsympE77}). Since the symps are the points of a long root geometry, and symplectic symps share a line, each symp symplectic or special to $\xi$ has at least one line in common with an opposite symp, yielding a point collinear to a unique point of $\xi$. 
\end{proof}

\subsection{Equator geometries and imaginary lines}\label{equatorgeom}
We already discussed the obvious relation between different parapolar spaces by taking point residuals, which is equivalent to taking the star of some vertex in the corresponding building. 

One notices that the point residuals of long root subgroup geometries are not long root subgroup geometries. In order to stay within the class of hexagonic geometries, one has to consider the so-called \emph{equator geometries} instead of point residuals. Essentially, instead of looking what is close to a given point, one considers what is in the middle of two given opposite points. General definitions and background can be found in \cite{HVM-MV}; here we content ourselves with the specific definitions for long root subgroup geometries, and more exactly in the exceptional cases $\mathsf{E_6,E_7,E_8}$.  A thorough reference is \cite{DSVM2}, and for metasymplectic spaces we refer to \cite{Lam-Mal:24}. 

So again let $\Gamma=(X,\cL)$ be either a metasymplectic space, or one of the geometries $\mathsf{E_{6,2}}(\K)$, $\mathsf{E_{7,1}}(\K)$ or $\mathsf{E_{8,8}}(\K)$, with $\K$ a field. Let $p,q$ be a pair of opposite points. Then the set 
$$E(p,q)=\{r\in X\mid \text{$r$ is symplectic to both $p$ and $q$}\}$$ is called an \emph{equator}. Note that each symplecton through $p$ contains a unique point of $E(p,q)$. We turn $E(p,q)$ into a point-line geometry (also denoted $E(p,q)$) by taking lines to be the subsets of points of $E(p,q)$ corresponding to the symplecta through $p$ sharing a given common maximal singular subspace. If $\Gamma$ is one of $\mathsf{E_{6,2}}(\K)$, $\mathsf{E_{7,1}}(\K)$ or $\mathsf{E_{8,8}}(\K)$, then the lines of $E(p,q)$ are just the lines of $\Gamma$ completely contained in $E(p,q)$. For $\Gamma\cong\mathsf{E_{7,1}}(\K)$ or $\Gamma\cong\mathsf{E_{8,8}}(\K)$ the equator geometry $E(p,q)$ is isomorphic to $\mathsf{D_{6,2}}(\K)$ or $\mathsf{E_{7,1}}(\K)$, respectively (both long root subgroup geometries). If $\Gamma$ is a metasymplectic space, then each line of $E(p,q)$ is a set of points contained in a common symplecton as the intersection of the perps of two opposite lines (opposite in the symplecton). 

For a given equator geometry $E(p,q)$, as defined in the previous paragraph, we call the points $p$ and $q$ \emph{poles} of $E(p,q)$. The set of all poles of a given equator geometry is called an \emph{imaginary line}. Imaginary lines carry the structure of a projective line over $\K$ with induced faithful stabiliser $\PSL_2(\K)$ in its natural action in the cases  $\mathsf{E_{6,2}}(\K)$, $\mathsf{E_{7,1}}(\K)$, $\mathsf{E_{8,8}}(\K)$ and $\mathsf{F_{4,1}}(\K,\A)$ (see \cite{Jan-Mal:25}). An imaginary line is determined by any two of its points $p,q$ and denoted $\cI(p,q)$.

For opposite points $p,q$, the imaginary line $\cI(p,q)$ can also be defined as the union of $\{p\}$ with the orbit containing $q$ of the long root subgroup with centre $p$. Such a subgroup consists of central elations with centre $p$, that is, collineations
fixing all points collinear and symplectic to $p$, and stabilising each line containing a point collinear to $p$ (see \cite{Jan-Mal:25}). The next lemma asserts that the latter condition is superfluous for type $\mathsf{E_7}$.  (A similar proof applies for the other long root subgroup geometries of exceptional types $\mathsf{E_6}$ and $\mathsf{E_8}$, but we only need the case of $\mathsf{E_7}$.)

\begin{lemma}\label{elationsE7}
A collineation $\theta$ of $\mathsf{E_{7,1}}(\K)$ is a central elation with centre $p$ if, and only if, $\theta$ fixes each point of $p^\perp\cup p^{\pperp}$.
\end{lemma}

\begin{proof}
We only have to show the sufficient condition. So suppose $\theta$ fixes each point of $p^\perp\cup p^{\pperp}$.
Let $x\perp p$ be an arbitrary point distinct from $p$ and collinear to $p$. It suffices to show that every line through $x$ is stabilised, in other words, that the residue at $x$ is pointwise fixed. Our assumption implies that in the residue at $x$, all points collinear or symplectic to the point corresponding to $p$ are fixed. Translated to the associated polar space $\Gamma$ of type $\mathsf{D_6}$ we have a collineation that stabilises all maximal singular subspaces of fixed type not disjoint from a given maximal singular subspace $U$ of the same type. Clearly, each point of $U$ is fixed. Also, every point of $\Gamma$ outside $U$ is the intersection of two lines, each with a point in $U$.  Finally, each line with exactly one point in $U$ is the intersection of two maximal singular subspaces of the same type as $U$ (and not disjoint from $U$, obviously). 
\end{proof}

Finally we mention a property of paras of $\mathsf{E_{7,1}}(\K)$. Note that paras of $\mathsf{E_{7,1}}(\K)$ are parapolar spaces isomorphic to $\mathsf{E_{6,1}}(\K)$, see \cite{Meu-Mal:22}. Also, a point $p$ not contained in para $\Pi$ is either collinear to all points of a maximal singular $5$-space of $\Pi$ (in which case we say that $p$ is \emph{close to} $\Pi$), or $p$ is contained in a unique para $\Pi'$ intersecting $\Pi$ in a symplecton; in this case the set of points $\Pi\cap\Pi'$ coincides with the set of points of $\Pi$ symplectic to $p$ and we say that $p$ is \emph{far from} $\Pi$.

\begin{lemma}\label{paraimline}
Given a para $\Pi$ of $\mathsf{E_{7,1}}(\K)$, let $q$ and $q'$ be two points far from $\Pi$ and symplectic to the same points of $\Pi$. If $q'$ is not contained in the unique para containing $q$ and sharing a symp with $\Pi$, then $q$ and $q'$ are opposite and the imaginary line $\cI(q,q')$ has a unique point in common with $\Pi$. 
\end{lemma}

\begin{proof}
Let $\Pi'$ be the unique para containing $q$ and intersecting $\Pi$ in a symp $\xi$. It follows from the discussion preceding the statement of the lemma that $q'$ is contained in a (unique) para $\Pi''$ containing $\xi$. Since $q'$ is far from $\Pi$, we see that $q'$ is opposite $\xi$ in $\Pi''$. Since $q$ is also opposite $\xi$ in $\Pi'$, an inspection of the standard apartment of type $\mathsf{E_7}$ as given in \cite{HVM-MV} reveals that $q$ and $q'$ are opposite. Moreover, it is easy to see that there is a unique central elation with centre $q$ mapping $\Pi''$ onto $\Pi$, which implies that $\Pi$ contains a unique member of $\cI(q,q')$.    
\end{proof}

We will also need another type of equator geometry. Let $\Pi$ and $\Pi'$ be two opposite paras in $\mathsf{E_{7,1}}(\K)$. Then, following \cite[Definition~6.6]{DSV2}, we define the \emph{equator geometry $E(\Pi,\Pi')$} as the geometry induced in the set of points which are close to both $\Pi$ and $\Pi'$. It is isomorphic to $\mathsf{E_{6,2}}(\K)$ (see \cite{DSV2} again). Since paras correspond to vertices of type $7$ in the associated building, it is readily seen that, with the notation of    \cite[Section~1.1]{PVMexc}, the set of fixed points of a homology $h_{\omega_7}(c)$, with $c\in\K\setminus\{0,1\}$, is exactly the union of two opposite paras together with their equator. The set of fixed paras can be seen as the following set of points of $\mathsf{E_{7,7}}(\K)$: if $p$ and $p'$ correspond to $\Pi$ and $\Pi'$, respectively, then the other fixed points are $(p^\perp\cap {p'}^{\pperp})\cup (p^{\pperp}\cap {p'}^\perp)$. The set $p^\perp\cap{p'}^{\pperp}$ will be referred to as a \emph{trace}.

\section{Class I automorphisms}\label{sec3}

We now initiate the proof of Theorem~\ref{nofixedchamber}. The first step is to analyse Class I automorphisms of the building $\sE_{8}(\K)$ (as defined in the introduction; recall that these automorphisms pointwise fix precisely a fully and isometrically embedded subspace isomorphic to $\mathsf{F_{4,1}}(\K,\H)$, for some quaternion division algebra $\H$ over $\K$). The main result of this section is summarised in the following existence and uniqueness theorem, which in particular shows that Class I automorphisms are domestic with opposition diagram~$\sE_{8;4}$. Recall the definition of the group $G(\K,\H)$ from Section~\ref{quatalg}.

\begin{theorem}\label{f4ine8}
Let $\K$ be a field and let $\H$ be a quaternion division algebra over $\K$ (recall that this includes inseparable quadratic extensions of degree $4$ in characteristic $2$). Then there exists a projectively unique embedded metasymplectic space $\Gamma:=\mathsf{F_{4,1}}(\K,\H)$ in the long root subgroup geometry $\Delta:=\mathsf{E_{8,8}}(\K)$, provided that in the inseparable case at least one symp of  $\Gamma$ is not contained in a singular subspace of $\Delta$. This embedding is isometric, and also convex in the building-theoretic sense. The pointwise stabiliser of $\Gamma$ is a group abstractly isomorphic to $G(\K,\H)$, each nontrivial member of which acts domestically on the building $\sE_8(\K)$ and does not fix any chamber. Both the fix and the opposition diagram of each nontrivial member of  the pointwise stabiliser is~$\mathsf{E_{8;4}}$.  
\end{theorem}

The condition in the inseparable case is necessary in view of the fact, proved in \cite{DSV2}, that each inseparable metasymplectic space with planes over $\K$ embeds in $\mathsf{E_{6,1}}(\K)$ such that symps are contained in singular subspaces, and $\mathsf{E_{8,8}}(\K)$ contains $\mathsf{E_{6,1}}(\K)$ as a subgeometry (in the residue of a line as a trace geometry \cite{HVM-MV}, or equivalently, as a para of an equator geometry). 

Throughout this section write $\Delta=\sE_{8,8}(\K)$ and $\Gamma=\sF_{4,1}(\K,\H)$ as in the statement of Theorem~\ref{f4ine8}, where $\mathbb{H}$ is a quaternion division algebra over $\K$  (possibly inseparable if the characteristic is 2). Moreover we assume that if $\H$ is inseparable over $\K$, then at least one symp of $\Gamma$ does not span a singular subspace of $\Delta$ (as in the statement of Theorem~\ref{f4ine8}).


\subsection{Quaternion dual polar spaces of rank 3 in $\mathsf{E_{7,7}}(\K)$}\label{sec4}

Throughout this section, let $\Gamma_1\cong \mathsf{C_{3,3}}(\mathbb{H},\K)$ be a point residual of $\Gamma$ (a dual polar space of rank $3$), and $\Delta_1\cong\mathsf{E_{7,7}}(\K)$ the corresponding point residual of~$\Delta$. Since $\Gamma$ is fully embedded in $\Delta$, $\Gamma_1$ is fully embedded in $\Delta_1$. 

A first step to proving Theorem~\ref{f4ine8} is to prove a related statement in the point residuals. Specifically, we prove:

\begin{prop}\label{prop:residual}
Let $\K$ be a field and let $\H$ be a quaternion division algebra over $\K$. Then there exists a projectively unique embedded dual polar space $\Gamma_1\cong\mathsf{C_{3,3}}(\mathbb{H},\K)$ in the point residual $\Delta_1\cong \mathsf{E_{7,7}}(\K)$, provided that in the inseparable case at least one symp of $\Gamma_1$ is not contained in a singular subspace of $\Delta_1$. Moreover this embedding is isometric, a group isomorphic to $G(\K,\H)$ acts on $\Delta_1$ as a collineation group, and the fixed point structure of every nontrivial member of it is precisely $\Gamma_1$. 
\end{prop}

Proposition~\ref{prop:residual} follows immediately from \cite[Proposition~6.17]{npvv} once we are able to show that every point residual of $\Gamma_1$ is a standard inclusion of the quaternion Veronese variety in the corresponding point residual of $\Delta_1$, and we achieve this goal in the following subsections (with the proof of Proposition~\ref{prop:residual} given at the end of Section~\ref{sec:global}).

\subsubsection{The embedding of symplecta in a point residual}

Let $\xi_1$ be a symp of $\Gamma_1$. Then $\xi_1$ is isomorphic to an orthogonal quadrangle $\mathsf{B_{2,1}}(\K,\mathbb{H})$ with standard representation in $7$-dimensional projective space over $\K$, and by  \cite[Lemma 3.19]{DSV2} $\xi_1$ is either (a) isometrically embedded in a unique symp $\zeta_1\cong\mathsf{D}_{6,1}(\K)$ of $\Delta_1$, or (b) is contained in a singular subspace of $\Delta_1$ (which has dimension at most $6$).  

In case (a), considering a representation of $\zeta_1$ in projective $11$-space $\PG(11,\K)$ it follows from \cite[Lemma~3.19]{DSV2} that $\xi_1$ arises as the intersection of $\zeta_1$ with a $7$-dimensional subspace $W$ of $\PG(11,\K)$. 

In case (b) since the singular subspaces of $\Delta_1$ have dimension at most 6,  $\xi_1$ can only embed in a singular subspace of $\Delta_1$ via a quotient of the universal embedding as a quadric and then $\kar\K=2$. However \cite[Proposition~3.18]{Pet-Mal:23} implies that $\xi_1$, viewed as a quadric in $7$-dimensional space, has empty nucleus whenever $\H$ is separable, and therefore case (b) can only arise in the inseparable case. 

Now, by assumption, in both cases there is at least one symp, say $\xi^*_1$, of $\Gamma_1$ that embeds isometrically in a symp $\zeta^*_1$ of $\Delta_1$. Thus, considering the natural embedding of $\zeta^*_1$ in $\PG(11,\K)$, there exists a subspace $W^*$ in $\PG(11,\K)$ such that $\xi^*_1=\zeta^*_1\cap W^*$. Note that, since $W^*$ has dimension $7$, every singular $5$-space of $\zeta_1^*$ intersects $\xi_1^*$ in exactly a line.  

Choose a point $p \in \xi_1^*$. Consider $\Delta_p:=\Res_{\Delta_1}(p)$, which is isomorphic to $\mathsf{E}_{6,1}(\K)$. We will work with the latter's standard embedding $\mathcal{E}_6(\K)$ in $\PG(26,\K)$.  Let $\Gamma_p$ be the point-line geometry in $\Delta_p$ with point set the set of lines of $\Gamma_1$ through $p$ and lines given by the symps of $\Gamma_1$ through $p$. Then $\Gamma_p$ is isomorphic to $\PG(2,\H)$, and we call its lines still symps of $\Gamma_p$.   
We denote by $\xi_2^*$ the symp of $\Gamma_p$ corresponding to $\xi_1^*$. Then $\xi_2^*$ embeds isometrically in the symp $\zeta_2^*$ of $\Delta_1'$ corresponding to $\zeta_1^*$ in such a way that, in the ambient projective $9$-space of $\zeta_2^*$, $\xi_2^*$ is the intersection of $\zeta_2^*$ with a $5$-space. It follows from a dimension argument that every singular $4$-space of $\zeta_2^*$ intersects $\xi_2^*$ in precisely a point. Hence $\xi_2^*$ is an ovoid of $\zeta_2^*$. 

We now show that, if there are symps of $\Gamma_1$ through $p$ which embed in singular subspaces (and so necessarily $\H$ is inseparable), then the latter are contained in $\zeta_2^*$. 

\begin{lemma}\label{sing}
With the above notation, suppose $\xi_1$ is a symp of $\Gamma_1$ containing $p$ and suppose $\xi_1$ embeds in a singular subspace. Let $\xi_2$ be the corresponding symp of $\Gamma_p$.  Then $\xi_2 \subseteq \zeta_2^*$.
\end{lemma}
\begin{proof}
Let $x$ denote the unique point in $\xi_2^* \cap\xi_2$.  Suppose for a contradiction that $\xi_2$ contains a point $x'$ not contained in $\zeta_2^*$. We start with an observation. Take any point $y$ of $\xi_2^*\setminus\{x\}$. Then the points $x$ and $y$ are not $\Delta_p$-collinear. As such, $x'$ is not $\Delta_p$-collinear to $y$, since it is already $\Delta_p$-collinear to $x$ and does not belong to $\zeta_2^*$ by assumption.

Set $S:=\<\xi_2\>$. Since $x'$ is $\Delta_p$-collinear to $x$, it is collinear to a $4'$-space $S'$ of $\zeta_2^*$ (note that $S \cap \zeta_2^* \subseteq S'$). Considering the tangent hyperplane at $x'$ to $\xi_2$, which meets $S \cap \zeta_2^*$ in a hyperplane of it (disjoint from $x$), it follows that there is a $3$-space $U$ in $S'$ which contains no points of $\xi_2$. Let $V$ be the unique $4$-space of $\xi_2^*$ containing $U$. Then $V$ contains a unique point $v$ of $\xi_2^*$. By the above observation, $v$ is not $\Delta_p$-collinear to $x'$, i.e., $v\in V\setminus U$. Consequently, $v$ is not $\Delta_p$-collinear to any point of $\zeta_2^*\cap \xi_2$, since such a point belongs to $S'\setminus U$ and $v^\perp\cap S'=U$. 

Now let $w$ be any point of $\xi_2^*\setminus\{x,v\}$. Then again $w$ and $x'$ are not $\Delta_p$-collinear by the above, and hence the symp $\xi_2'$ of $\Gamma_p$ they determine embeds isometrically in a symp $\zeta_2'$ of $\Delta_p$. Since $w$ is $\Delta_p$-collinear to a $3$-space of $S'={x'}^{\perp}\cap \zeta_2^*$, it follows that $\zeta_2'$ meets $\zeta_2^*$ in a $4$-space $S^*$. 
The point $v$ is $\Delta_p$-collinear to a $3$-space $U^*$ of $S^*$. The unique $4'$-space $V'$ of $\zeta_2'$ containing $U^*$   contains a unique point $v'$ of $\xi_2'$. Now, since $\<v,U^*\>$ and $V'$ are $4'$-spaces sharing a $3$-space, they determine a singular $5$-space of $\Delta_p$. So $v'$ and $v$ are $\Delta_p$-collinear and hence the symp $\xi_{v,v'}$ of $\Gamma_p$ they determine embeds in a singular subspace.
Let $z$ be the unique intersection point of $\xi$ and $\xi_{v,v'}$. Then $z$ is $\Delta_p$-collinear to both $v$ and $x$. Therefore, $z$ belongs to $\zeta_2^* \cap \xi_2$ and is $\Delta_p$-collinear to $v$. However, as mentioned above, $v$ was chosen such that it is not $\Delta_p$-collinear to any point of $\xi_2^* \cap \xi$. We conclude from this contradiction that $\xi_2\subseteq \zeta_2^*$.
\end{proof}

\begin{lemma}\label{Q^*}
With the above notation, suppose some symp $\xi_2\neq\xi_2^*$ of $\Gamma_p$ is contained in $\zeta_2^*$. Then all symps of $\Gamma_p$ are contained in $\zeta_2^*$.
\end{lemma}
\begin{proof}
Let $x$ denote the unique point in $\xi_2^* \cap\xi_2$. Take any symp $\xi'_2$ of $\Gamma_p$ not through $x$. Then $\xi_2'$ meets $\xi_2^*$ and $\xi_2$ in distinct points $z^*$ and $z$, respectively. If $z$ and $z^*$ are $\Delta_p$-collinear, then $\xi'_2$ embeds in a singular subspace and hence Lemma~\ref{sing} implies that $\xi'_2\subseteq \zeta_2^*$. If $z$ and $z^*$ are not $\Delta_p$-collinear, then $\xi_2'$ embeds isometrically in the unique symp of $\Delta_p$ containing $z$ and $z^*$, being $\zeta_2^*$. So in both cases, $\xi_2'\subseteq \zeta_2^*$. Interchanging the roles of $\xi_2'$ and $\xi_2$, the same conclusion is obtained for any symp of $\Gamma_p$ containing $x$. The lemma follows.
\end{proof}

%

%

\begin{lemma}\label{embres}
With the above notation, and with the hypothesis of Theorem~\ref{f4ine8}. each symp of $\Gamma_p$ embeds isometrically in a symp of $\Delta_p$ and this correspondence is injective (that is, no two symps of $\Gamma_p$ embed in the same symp of $\Delta_p$).
\end{lemma}

\begin{proof}
Suppose for a contradiction that either there is a symp of $\Gamma_p$ which embeds in a singular subspace of $\Delta_p$, or all symps of $\Gamma_p$ embed isometrically but there are two distinct symps which embed in the same symp of $\Delta_p$. Note that each symp which embeds isometrically plays the same role as $\xi_2^*$. In both cases, it follows from Lemmas~\ref{sing} and~\ref{Q^*} that all symps of $\Gamma_p$ are contained in $\zeta_2^*$. 

Consider a point $x$ of $\Gamma_p$ not contained in $\xi_2^*$. Then each maximal singular subspace of $\zeta_2^*$ containing $x$, contains a unique point of $\xi_2^*$. Let $U$ be such a maximal singular subspace. Then we can pick a maximal singular subspace $U'$ intersecting $U$ in a $3$-space that does not contain either of the two points $x$ or $U\cap\xi_2^*$. We see that $U'\cap\xi_2^*$ is not collinear to $x$.    This implies 
that there 
is at least one symp of $\Gamma_p$ through $x$ which embeds isometrically. Since we can find two maximal singular subspaces through $x$ only having $x$ in common, there are at least two which do not.

So we may select a symp $\xi_2$ of $\Gamma_p$ through $x$ which embeds in a singular subspace $U$ of $\zeta_2^*$, and one, say $\xi_2'$, also through $x$, which embeds isometrically. We may assume that $U$ is a maximal singular subspace, possibly not generated by $\xi_2$. However, there is a unique hyperplane of $\<\xi_2\>$ tangent to $\xi_2$ at $x$ (as $\xi_2$ comes from an embedded polar space of rank $2$---a symp of $\Gamma_1$). We can easily extend this hyperplane to a $3$-space $S$ of $U$ intersecting $\xi_2$ in only $x$. Let $V$ be the unique other maximal singular subspace of $\zeta_2^*$ containing $S$. Then $V$ contains a unique point $v$ of $\xi_2^*$. Since $v$ is not collinear to $\xi_2\cap\xi_2^*$, as $\xi^*$ is an ovoid of $\zeta_2^*$, the point $v$ is not contained in $S$. But then $v^\perp\cap\xi_2\subseteq v^\perp\cap U=S$, implying that there is only one symp of $\Gamma_p$ through $v$ which embeds in a singular subspace (and it contains $x$; hence $v\notin \xi'_2$). Interchanging the roles of $(x,\xi_2^*)$ and $(v,\xi_2')$, we deduce from the first paragraph that there are at least two symps through $v$ that are embedded in a singular subspace, a contradiction.

We obtain that our initial assumption is wrong, from which the lemma follows.
\end{proof}

\subsubsection{The global situation in any point residual}\label{sec:global}
 A \emph{classical} ovoid $O$ of a quadric is an ovoid that is the intersection of the quadric with a subspace $S$ of the ambient projective space. The \emph{nucleus} of $O$ is the intersection of all tangent hyperplanes to $O$ in $S$. Since we are dealing with ovoids that are residues of quadrics $Q$ of Witt index $2$ (quadrics that contain lines but no planes: the symps of $\Gamma_1$), and each embedding of $Q$ in a projective space is obtained from a quotient of the standard embedding as a quadric, the embeddings of $O$ that will occur here are also quotients of their standard representation as a quadric of $S$. 
 
 By the previous sections, we have shown the following proposition.

\begin{prop}\label{recap}
Let $\Gamma_p=(X,\cL)$ be a point-line geometry isomorphic to a projective plane, contained in $\mathsf{E_{6,1}}(\K)$ such that each member of $X$ is a point of $\mathsf{E_{6,1}}(\K)$ and such that each member of $\cL$ is either a classical ovoid of some symp, or the quotient of a classical ovoid contained in some singular subspace. Assuming $X$ itself is not contained in a singular subspace, each member of $\cL$ is then an ovoid of a symp of $\mathsf{E_{6,1}}(\K)$, and no two distinct members of $\cL$ are contained in the same symp of $\mathsf{E_{6,1}}(\K)$. In particular, each pair of points of $X$ is symplectic.
\end{prop}

Next we want to show that, with the notation of \cref{recap}, and with reference to \cref{quatalg}, the set $X$, viewed as a subset of points of $\PG(26,\K)$ is a quaternion Veronese variety in some $14$-dimensional subspace of $\PG(26,\K)$.    The following lemma paves the way to achieve that.


\begin{lemma}\label{far}
Let $q$ be a point of $\Gamma_p$ and $\xi_2,\xi_2'\in\cL$ symps of $\Gamma_p$. Let $\zeta_2$ and $\zeta_2'$ be the respective corresponding symps of $\Delta_p$. Then:
\begin{compactenum}[$(i)$]
\item If $q\notin \xi_2$, then $q$ is far from $\zeta_2$, that is, $q^{\perp} \cap \zeta_2=\varnothing$.
\item If $\xi_2\neq \xi_2'$ then $\zeta_2$ and $\zeta_2'$ share a unique point (namely the point $\xi_2 \cap \xi_2'$).
\end{compactenum}
\end{lemma}
\begin{proof}\begin{compactenum}[$(i)$]
\item Suppose for a contradiction that $q$ is close to $\zeta_2$, that is, $q^\perp\cap \zeta_2$ is a $4'$-space $U$. Since $\xi_2$ is an ovoid of $\zeta_2$, there is some point $x\in\xi_2\cap U$, contradicting \cref{recap} that ensures that $q$ and $x$ are symplectic and not collinear.
\item
Suppose for a contradiction that $\zeta_2\cap \zeta_2'$ is a $4$-space $W$. Select $y\in\xi_2'\setminus W$. Then $y\notin\xi_2$ is close to $\zeta_2$, contradicting $(i)$. \qedhere \end{compactenum} 
\end{proof}

We can now show:
\begin{prop}\label{prop:veronese}
The set $X$, viewed as a subset of the ambient projective space $\PG(26,\K)$ of $\mathcal{E}_6(\K)$, is a quaternion Veronese variety in a subspace of dimension $14$.
\end{prop}
\begin{proof}
If two symps $\zeta_2,\zeta_2'$ of $\mathsf{E_{6,1}}(\K)$ intersect in a point, then also the subspaces $\<\zeta_2\>$ and $\<\zeta_2'\>$ intersect in a point (see \cite[\S7.1]{Sch-Mal:17}). Hence, by \cref{far}, for each pair of distinct symps $\xi_2,\xi_2'$ of $\Gamma_p$, the subspaces $\<\xi_2\>$ and $\<\xi_2'\>$ intersect in a unique point. Now the assertion follows from \cite[Main Result~4.3]{DSVM}. 
\end{proof}

We now give the proof of Proposition~\ref{prop:residual}. 

\begin{proof}[Proof of Proposition~\ref{prop:residual}]
By Proposition~\ref{prop:veronese} and \cite[Remark~5.3]{npvv} we have a standard inclusion of the quaternion Veronese variety in $\mathsf{E_{6,1}}(\K)$ in the point residual at $p$. Let $q$ be any point of $\Gamma_1$ collinear to $p$. Since the symps of $\Gamma_1$ through $pq$ are isometrically embedded, we can interchange the roles of $p$ and $q$ and obtain that all symps of $\Gamma_1$ through $q$ are isometrically embedded. Connectivity then implies that all symps of $\Gamma_1$ are isometrically embedded in $\Delta_1$ and in each point residual we have a standard inclusion of the quaternion Veronese variety in $\mathsf{E_{6,1}}(\K)$.

Now  \cite[Proposition~6.17]{npvv} shows existence and uniqueness of $\Gamma_1$ in $\Delta_1$. Moreover we obtain that $\Gamma_1$ is isometrically embedded in $\Delta_1$ (this follows from \cite[Lemma~6.16]{npvv}), and a group isomorphic to $G(\K,\H)$ acts on $\Delta_1$ as a collineation group, and the fixed point structure of every nontrivial member of it is precisely $\Gamma_1$. 
\end{proof}

%
%
%
%

We also note the following properties:
\begin{cor}\label{propres}
\begin{compactenum}[$(i)$] \item Every point of $\Delta_1$ is collinear to a least one point of $\Gamma_1$.
\item Each line of $\Delta_1$ carrying exactly one point $p$ of $\Gamma_1$ is contained in exactly one symp of $\Delta_1$ that contains a symp of $\Gamma_1$.
\end{compactenum}
\end{cor}

\begin{proof}
$(i)$ is the last assertion of Proposition 6.2 of  \cite{npvv}, and $(ii)$ follows from Lemma 4.11 of  \cite{npvv} by considering the point residual at $p$. 
\end{proof}

\subsection{The embedding of $\Gamma$ in $\Delta$}

We return to our main goal of proving Theorem~\ref{f4ine8}. We first prove that the embedding of $\Gamma$ in $\Delta$ is isometric. 

\begin{lemma}
The metasymplectic space $\Gamma$ is isometrically embedded in $\Delta$.
\end{lemma}
\begin{proof}
Clearly, for two collinear points $p$ and $q$ of $\Gamma$ we have that, if the point residual at $p$ contains an isometrically embedded symp, then so does the point residual at $q$ (because by Proposition~\ref{prop:residual} all symps through $p$ are isometrically embedded). By connectivity, it follows that all symps and all point residuals are isometrically embedded. Hence symplectic points in $\Gamma$ are also symplectic in $\Delta$. Let $p$ and $q$ be special points of $\Gamma$. Then in the point residual of their centre, say $c$, the lines $cp$ and $cq$ correspond to opposite points of $\Res_\Gamma(c)$. Since point residuals are isometrically embedded, this implies that $cp$ is opposite in $cq$ in $\Res_\Delta(c)$, and so $\{p,q\}$ is a special pair of $\Delta$. 

Finally suppose $p$ and $q$ are opposite in $\Gamma$. Then \cref{spspop}, which holds in both $\Gamma$ and $\Delta$,  implies that $p$ is opposite $q$ in $\Delta$, and the lemma is proved.   
\end{proof}

Our next goal is to show that $\Gamma$ is also convex in $\Delta$, in the building-theoretic sense, that is, closed under projections. Since the embedding is isometric, it is easy to see by extending galleries in $\Gamma$ that this is equivalent to the property that every apartment in $\Delta$ containing two given opposite chambers of $\Gamma$, viewed as flags of $\Delta$, contains the unique apartment of $\Gamma$ containing these two chambers.  

\begin{lemma}\label{convexemb}
The embedding of $\Gamma$ in $\Delta$ is convex.
\end{lemma}

\begin{proof}
Consider two opposite chambers $C=\{p,L,\alpha,\xi\}$ and $D=\{q,M,\beta,\zeta\}$ of $\Gamma$, with $p,q$ points, $L,M$ lines, $\alpha,\beta$ planes, and $\xi,\zeta$ symps of $\Gamma$. Let $\sigma$ and $\omega$ be the symps of $\Delta$ containing $\xi$ and $\zeta$, respectively. We first show that the projection of $C$ onto each of the members of $D$ coincide whether considered in $\Gamma$ or $\Delta$. Note that, since the embedding is isometric, collinearity of points, but also opposition of arbitrary flags in $\Gamma$ is the same as in $\Delta$. 
\begin{compactenum}[$(i)$]
\item Set $\{q,M',\beta',\omega'\}=\proj_q(C)$ in $\Delta$.   The line $M'$ is the unique line in $\Delta$ through $q$ containing a point collinear to some point of $L$. The plane $\beta'$ is the unique plane through $q$ containing a line all points of which are collinear to some point of $\alpha$. Finally, the symp $\omega'$ is the unique symp of $\Delta$ through $q$ intersecting $\sigma$ (necessarily in a unique point, say $t$). If $\zeta'$ is the unique symp of $\Gamma$ through $q$ intersecting $\xi\subseteq\sigma$, then this intersection point must clearly be $t$ and $\zeta'\subseteq\omega'$. This implies that  $\{q,M',\beta',\zeta''\}=\proj_q(C)$ in $\Gamma$.
\item Similarly, since being a special pair is the same in $\Gamma$ as in $\Delta$, and since projecting planes and symplecta from a line to an opposite line is determined by the mutual position of certain points in the same way in $\Gamma$ as in $\Delta$ by \cref{projline}, we conclude that $\proj_M(C)$ is the same in $\Delta$ as in $\Gamma$. 
\item In the same vein,  $\proj_\beta(C)$ is the same whether considered in $\Gamma$ or in $\Delta$, now using \cref{projplane'}.  
\item The last case is also easy:  since projection from a symp to an opposite is just given by being symplectic, it immediately follows that  $\proj_\zeta(C)$ in $\Gamma$ is the pointwise restriction of $\proj_\omega(C)$ to $\xi$.  
\end{compactenum}    
Now since $\proj_q(C)$ is, in $\Gamma$, opposite $\proj_p(D)$, and likewise for $\proj_M(C)$ and $\proj_L(D)$, for $\proj_\beta(C)$$\proj_\alpha(D)$, and for $\proj_\zeta(C)$ and $\proj_\xi(D)$, the same is true in $\Delta$, and we can apply the foregoing to these pairs of opposite chambers. Continuing like this we obtain the whole apartment in $\Gamma$ spanned by $C$ and $D$, and that this is a convex set in $\Delta$, also contained in the convex hull of $C$ and $D$ and hence in an apartment of $\Delta$.  
\end{proof}

Using the group $G(\K,\H)$ pointwise fixing a point residual of $\Gamma$ inside the corresponding point residual of $\Delta$, and the identity group, pointwise fixing a plane of $\Gamma$, we can apply  \cite[Proposition~4.16]{Tits:74}, in just the same way as done in the proof of  \cite[Proposition~4.16]{DSV2}, or \cite[Proposition~6.17]{npvv}
, to obtain:
\begin{prop}
Suppose the field $\K$ admits a quaternion division algebra (separable or inseparable) $\H$. Then $\mathsf{E_{8,8}}(\K)$ admits a full embedding of $\mathsf{F_{4,1}}(\K,\H)$, unique up to a projectivity if we assume in the inseparable case that at least one symp---and then automatically each symp---is not contained in a singular subspace. Also, this embedding arises as the fixed point set of each nontrivial collineation of a group of collineations isomorphic to the group $G(\K,\H)$. The embedding is isometric and convex and has fix diagram $\sE_{8;4}$.
\end{prop}

This shows a large part of \cref{f4ine8}. It remains to show that every collineation $\theta$ of $\Delta$ pointwise fixing $\Gamma$ is domestic with opposition diagram $\mathsf{E_{8;4}}$ if $\theta$ is nontrivial. We prove this in the following subsection.

\subsection{Domesticity of Class I automorphisms}

In the following lemma we require another type of equator geometry (defined in \cite[Definition~3.6 ]{DSVM2}). Let $\Delta_1=\mathsf{E_{7,7}}(\K)$ and let $\xi,\xi'$ be opposite symps. Then the \textit{equator geometry $E(\xi,\xi')$} is the set of points of $\Delta_1$ collinear to a $5'$-space of $\xi$, and also to one of $\xi'$. By  \cite[Lemma~3.7]{DSVM2}, $E(\xi,\xi')$ is isomorphic to $\mathsf{D_{6,6}}(\K)$. 
Also for the lemma, a \emph{spread (of lines)} of a polar space is a partition of its point set in lines.

%

\begin{lemma}\label{inres}
Let $\Gamma_1=(X,\cL)\cong\mathsf{C_{3,3}}(\H,\K)$ be a quaternion dual polar space isometrically embedded in $\Delta_1\cong\mathsf{E_{7,7}}(\K)$ and let $\theta$ be a collineation of $\Delta_1$ pointwise fixing $\Gamma_1$. Then
\begin{compactenum}[$(i)$]
\item No point is mapped onto a collinear or opposite one, and every point is collinear to at least one fixed point.
\item If a point is mapped onto a symplectic one, then the corresponding symp is fixed by~$\theta$.
\item Let $\xi_1$ be a symp of $\Delta_1$ mapped to an opposite symp under the action of $\theta$. Then  $X\cap E(\xi_1,\xi_1^\theta)$ is the point set of an induced quadratic dual polar space $\Gamma_1'\cong\mathsf{C_{3,3}}(\L,\K)$, with $\K\leq\L\leq\H$. The set of symps of $\Gamma_1'$ is in natural bijective correspondence with the set of lines of $\xi_1$ mapped by $\theta$ to their projection onto $\xi_1^\theta$, and forming a spread of $\xi_1$. 
\end{compactenum}
\end{lemma}

\begin{proof}
We first observe that, by  \cite[Proposition~6.2]{npvv}
, $\theta$ is domestic with opposition diagram $\mathsf{E_{7;4}}$. Then $(i)$ follows from this and \cite[Lemma~6.7]{npvv}
. The same proposition in \cite{npvv} implies that every point $p$ is collinear to some fixed point $x$. Hence $x$ is contained in the symp $\xi(p,p^\theta)$. Now $(ii)$ follows from \cite[Corollary~6.1]{npvv}. Now we show $(iii)$.

Pick an arbitrary point $p\in\xi_1$. By $(i)$, $p$ is mapped onto a symplectic point. Set $\zeta_1:=\xi(p,p^\theta)$.  Then $\zeta_1$ is fixed under $\theta$. By Definition~3.7, Proposition~3.8 and Lemma~4.3 in \cite{npvv}, every maximal singular subspace of $\zeta_1$ contains a pointwise fixed line. 
Also, $\zeta_1\cap\xi_1$ is at least a line, but if it were a $5$-space, then each intersection point of $\xi_1^\theta$ and $\zeta_1$ would be collinear with a $5'$-space of $\xi_1$, contradicting the fact that $\xi_1$ is opposite $\xi_1^\theta$. Hence $\zeta_1\cap\xi_1$ is a line $L$.  Obviously $L^\theta=\zeta_1\cap \xi_1^\theta$, and $L$ is $\zeta_1$-opposite $L^\theta$ as each point of $L\subseteq\xi_1$ is collinear to at most one point of $L^\theta\subseteq\xi_1^\theta$, and hence to exactly one point since $L,L^\theta\subseteq\zeta_1$. By $(i)$, $p^\theta\in L^\theta$ is not collinear to $p$ and so its projection $p'$ onto $\xi_1$ belongs to $L\setminus\{p\}$. It follows that the mapping $\theta_{\xi_1}:{\xi_1}\to{\xi_1}:p\mapsto p'$ is point-domestic without fixed points. By \cite[Theorem 8]{PVMclass}, the set of fixed maximal singular subspaces forms the point set of a dual polar space isomorphic to $\mathsf{C_{3,3}}(\L,\K)$, for some field $\L$ quadratic over $\K$. Now, we claim that the fixed point structure of $\theta$  in $E({\xi_1},\xi_1^\theta)$ is isomorphic to the fixed point structure of $\theta_{\xi_1}$ formed by the fixed $5'$-spaces of $\xi_1$. Indeed, each fixed point $p$ in $E(\xi_1,\xi_1^\theta)$ gives rise to a $5'$ space $W=p^\perp\cap\xi_1$ in $\xi_1$, and a unique $5'$-space $W'=p^\perp\cap\xi_1^\theta$ of $\xi_1^\theta$, which coincides with $W^\theta$. Suppose for a contradiction that the image of $W'$ under the projection of $\xi_1^\theta$ onto $\xi_1$ is distinct from $W$. Then some point $x'$ of $W'$ is collinear to a point $x\in\xi_1\setminus W$. Consequently, by \cref{pointsympE77}, $x'$ is opposite every point of $W\setminus x^\perp$, which is nonempty. This contradicts $x'\perp p\perp x$ and hence $W$ is fixed under $\theta_{\xi_1}$. Conversely, suppose a $5'$-space $W$ of $\xi_1$ is fixed under $\theta_{\xi_1}$. Let $\widetilde{W}$ be the unique $6$-space containing $W$. Then $\widetilde{W}$ contains a unique point $w\in E(\Sigma,\Sigma^\theta)$. Then, by the above argument, $w^\perp\cap\Sigma^\theta=W^\theta$. The $6$-space generated by $w$ and $W^\theta$ is the unique $6$-space containing $W^\theta$ and is hence the image under $\theta$ of $\widetilde{W}$. It follows that $w^\theta\perp w$ and hence, by $(i)$, $w=w^\theta$. It remains to show that the maximal subspaces of $\Sigma$ fixed under $\theta_{\xi_1}$ are really $5'$-spaces (and not $5$-spaces). By the foregoing, this is equivalent to showing that $E(\xi_1,\xi_1^\theta)$ contains at least one fixed point under $\theta$.  Let $U$ be a $5$-space of $\zeta_1$ through $L$. Then, as noted above, $U$ contains a pointwise fixed line $M$. Let $y\in M$. Then $y\perp L$ and hence $y^\perp\cap\xi_1$ is a $5'$-space. Consequently $y^\perp\cap\xi_1^\theta$ is also a $5'$-space and $y\in X\cap E(\xi_1,\xi_1^\theta)$. The claim is now completely proved. 

Hence the fixed point structure of $\theta$ in $E(\xi_1,\xi_1^\theta)$, which equals  $X\cap E(\xi_1,\xi_1^\theta)$, is isomorphic to $\mathsf{C_{3,3}}(\L,\K)$. Obviously, since it is contained in $X$, we find that $\L\leq\H$. 

The last assertion of $(iii)$ now also follows from \cite[Theorem 8]{PVMclass}. 
\end{proof}

Now we can show the main result of this subsection.

\begin{prop}\label{classIisdom}
Let $\theta$ be a nontrivial collineation of $\Delta\cong\mathsf{E_{8,8}}(\K)$ pointwise fixing precisely an isometrically embedded metasymplectic space $\Gamma=(X,\cL)\cong\mathsf{F_{4,1}}(\K,\H)$, with $\H$ a quaternion division algebra over $\K$, possibly inseparable. Then $\theta$ is domestic with opposition diagram~$\sE_{8;4}$. 
\end{prop}

\begin{proof}
Suppose, for a contradiction, that $C$ is a chamber mapped to an opposite chamber by~$\theta$. Then its point $p$ is mapped onto an opposite point. Clearly $p$ is not collinear to any point of $\Gamma$. We claim it is symplectic to at least one point of $\Gamma$. Indeed, suppose not.  Considering an arbitrary line of $\Gamma$, we see that at least one point $x$ of $\Gamma$ is special to~$p$.   
Now, by \cref{inres}$(i)$
,  each point of $\Res_{\Delta}(x)$ is collinear to some fixed point in this residue, that is, each line through $x$ lies in a plane containing a fixed line.  Hence the line $cx$, with $c=\mathfrak{c}(x,p)$, is contained in a plane $\pi$ containing a fixed line $L$ through $x$. By the second axiom of hexagonic geometry (\cref{KasShu}), there is a point $y\in L$ symplectic to $p$ and $y$ belongs to $\Gamma$. The claim is proved. 

So we established a point $y\pperp p$ with $y$ in $\Gamma$. Note that $y=y^\theta\pperp p^\theta$, and so $y\in p^{\pperp}\cap (p^\theta)^{\pperp}=E(p,p^\theta)$.  Set $\xi:=\xi(p,y)$. Let $z$ be any point of $p^\perp\cap y^\perp$. Since $z\perp p$, we know that $z$ does not belong to $\Gamma$. But   \cref{inres}$(i)$ yields a plane $\alpha\ni x$ containing a line $K\ni y$ of $\Gamma$. Now \cref{inres}$(ii)$ yields a (unique) symplecton $\zeta$ of $\Delta$ containing $\alpha$ and fixed under $\theta$. It follows from \cref{spspop}, the assumption that $p$ is opposite $p^\theta$, and the fact that both $p^\perp\cap \zeta$ and $(p^\theta)^\perp\cap\zeta$ contain a line, that $p^\perp\cap \zeta$ and $(p^\theta)^\perp\cap\zeta$ are $\zeta$-opposite lines, say $R$ and $R^\theta$, respectively. Since $R^\perp\cap (R^\theta)^\perp$ is isomorphic to $\mathsf{D_{5,1}}(\K)$, for which a classical ovoid in the perp of a line canonically lives in a $3$-space and hence is $2$-dimensional,  \cite[Proposition~3.8]{npvv} 
implies that $X\cap R^\perp\cap(R^\theta)^\perp$ is a polar space, containing~$y$, canonically isomorphic to $\mathsf{B_{3,1}}(\K,\L)$, for some quadratic extension $\L$ of $\K$ contained in $\H$.  This quadratic extension is determined by the point residual at $p$ of $X\cap R^\perp\cap(R^\theta)^\perp$, which is isomorphic to $\mathsf{B_{2,1}}(\K,\L)\cong \mathsf{C_{2,2}}(\L,\K)$. 

We claim that every symp $\zeta'$ of $\Delta$ containing two symplectic fixed points $y_1,y_2$ in $E(p,p^\theta)$ can play the role of $\zeta$ in the previous paragraph. Indeed, we may without loss of generality assume $y_2=y\in  \zeta'$. Since $y\pperp y_1$, the symps $\xi(y,p)$ and $\xi(y_1,p)$ intersect in a line $K$ (this follows from the isomorphism between $\Res_\Delta(p)$  and $E(p,p^\theta)$, see \cite{HVM-MV}). The unique points $z_1$ and $z_2$ of $K$ collinear to $y_1$ and $y$, respectively, coincide as both points are not opposite $p^\theta$ (by \cref{spspop}). Hence $z_1=z_2$ can play the role of $z$ in the previous paragraph, and it is now easy to find a plane $\alpha$ in $\zeta'$ containing a fixed line. This proves the claim.

Now note that $R\subseteq p^\perp\cap y^\perp\subseteq\xi$. Exhausting all points $z$ of $p^\perp\cap y^\perp$ as above, we obtain a line spread of $p^\perp\cap y^\perp$. Now Lemma~\ref{inres} implies that the fixed point structure $Y$ of $\Res_{\Delta}(y)$ symplectic to both $p$ and $p^\theta$ (which is equivalent to each point of $Y$ being close to both $\xi$ and $\xi^\theta$, which is equivalent to $Y$ belonging to the equator with poles the symps  $y^\perp\cap\xi$ and $y^\perp\cap\xi^\theta$ of $\Res_{\Delta}(y)$) is a dual polar space isomorphic to $\mathsf{C_{3,3}}(\L',\K)$, with $\K\leq\L'\leq\H$. Moreover, the last assertion of Lemma~\ref{inres}$(iii)$ implies that the point residual at $p$ of $X\cap R^\perp\cap(R^\theta)^\perp$ is a symp of that dual polar space, implying that  $\L$ and $\L'$ coincide.

Since each fixed point symplectic to $p$ is also symplectic to $p^\theta$, we have shown that the fixed points in $E(p,p^\theta)$ form an isometrically embedded geometry $\Gamma'=(Z,\cM)$ with the following properties.
\begin{compactenum}[$(i)$]
\item Each symplectic pair of points is contained in a unique convex subgeometry isomorphic to  $\mathsf{B_{3,1}}(\K,\L)$.
\item The point residual of $\Gamma'$ at each point $z\in Z$ is isomorphic to $\mathsf{C_{3,3}}(\L,\K)$. 
\end{compactenum}

It follows rather easily that $\Gamma'$ is convexly (in the building-theoretic sense) embedded in~$\Gamma$, because it is isometrically embedded and projections of objects onto other objects are determined by the mutual positions between the points of these objects, see the proof of \cref{convexemb}. Since $\Gamma'$ contains opposite chambers of $\Gamma$, it follows that it corresponds to a subbuilding of the building corresponding to $\Gamma$. From the two properties mentioned above, we then deduce that $\Gamma'$ is a parapolar space isomorphic to $\mathsf{F_{4,1}}(\K,\L)$, isometrically embedded in $E(p,p^\theta)$. 

We now consider the automorphism $\theta_p$ of $E(p,p^\theta)$ mapping each point $x$ onto $\xi(p^\theta,x^\theta)\cap E(p,p^\theta)$. Clearly, $\theta_p$ pointwise fixes $Z$, and so by \cite[Proposition~7.1]{npvv} 
$\theta_p$ is domestic with opposition diagram $\mathsf{E_{7;3}}$. It follows, using Theorem 3.28 and Proposition 3.29 of \cite{Tits:74}, that $C$ is not mapped to an opposite, a contradiction.

Thus $\theta$ is domestic, and since $\theta$ fixes no chamber it has opposition diagram~$\sE_{8;4}$ (by the classification of automorphisms with opposition diagram $\sE_{8;1}$ and $\sE_{8;2}$ in \cite{PVMexc}). 
\end{proof}

We note down an interesting consequence of the previous arguments, summarising the situation. We use the notation of the previous proof. Note that $\theta$ does not necessarily stabilise $E(p,p^\theta)$, but $\theta_p$ of course does. 
\begin{cor}\label{coruniclass}
Let $p$ be a point such that $p^\theta$ is opposite $p$. Then the fix structure of $\theta$ and $\theta_p$ in $E(p,p^\theta)$ is a parapolar space isomorphic to $\mathsf{F_{4,1}}(\K,\L)$, isometrically and fully embedded, with fix diagram $\mathsf{E_{7;4}}$. 
\end{cor}

\section{Class II automorphisms}\label{sec:equator}




In this section we use the notation of \cref{fixingchamber}. The main theorem of this section is as follows, giving an explicit description, condition for existence, and domesticity of Class II automorphisms.

\begin{theorem}\label{thm:equatordomestic}
Every automorphism pointwise fixing an equator geometry in $\sE_{8,8}(\K)$ is domestic. Moreover, every automorphism pointwise fixing an equator geometry and fixing no chamber is conjugate to an element of the form $x_{\varphi}(a)x_{-\varphi}(1)$ with $a\in\K^{\times}$ and $aX^2+aX-1$ irreducible over~$\K$, and has opposition diagram~$\sE_{8;4}$. 
\end{theorem}

\begin{proof}
The points of the long root geometry $\sE_{8;8}(\K)$ can be identified with the cosets $G_0/P$ where $P=\bigsqcup_{w\in W_{\sE_7}}BwB$. Let $M$ denote the set of minimal length double coset representatives in $W_{\sE_7}\backslash W/W_{\sE_7}$. Then $M$ consists of precisely $5$ elements, $e, s_8,w_1,w_2,s_{\varphi}$ (arranged in increasing length, with $\varphi$ the highest root), and $G_0=\bigsqcup_{w\in M}PwP$. Points $g_1P$ and $g_2P$ are collinear (respectively symplectic, at special distance, opposite) if, and only if, $g_1^{-1}g_2\in Ps_8P$ (respectively $g_1^{-1}g_2\in Pw_1P$, $Pw_2P$, $Ps_{\varphi}P$). 

Since $G_0$ acts transitively on pairs of opposite points, we may assume, up to conjugation, that the poles of the equator geometry are $P$ and $s_{\varphi}P$. As described in the paragraph before Lemma~\ref{elationsE7}, the imaginary line $\cI(P,s_{\varphi}P)$ is the union of $\{P\}$ with the orbit containing $s_{\varphi}P$ of the long root subgroup with centre $P$. 
Thus
$$
\mathcal{I}(P,s_{\varphi}P)=\{P\}\cup\{x_{\varphi}(a)s_{\varphi}P\mid a\in\K\},
$$
and the stabiliser of $\mathcal{I}(P,s_{\varphi}P)$ (and hence the stabiliser of $E(P,s_{\varphi}P)$) is
$$
G_{\varphi}=\langle U_{\varphi},U_{-\varphi}\rangle\cong \mathsf{PSL}_2(\K).
$$ 
Thus it is sufficient to prove that every element of the group $G_{\varphi}$ acts domestically on~$\Delta$. Let $G_{\varphi}'=\langle G_{\varphi},H_{\alpha_8}\rangle$, where $H_{\alpha_8}=\{h_{\alpha_8}(t)\mid t\in \K^{\times}\}$. We first claim that each element $g\in G_{\varphi}$ can be conjugated, by an element of $G_{\varphi}'$, to one of the following forms:
$$
g_1=h_{\varphi}(t),\quad g_2=x_{\varphi}(1),\quad g_3=x_{\varphi}(1)h_{\varphi}(-1),\quad g_4=x_{\varphi}(a)x_{-\varphi}(1)
$$ 
where $t\in\K^{\times}$ and $a\in\K$. 

By the Bruhat decomposition in $\langle U_{-\varphi},U_{\varphi}\rangle$, each element $g$ can be written, up to conjugation in $G_{\varphi}$, as either $g=x_{\varphi}(a)h_{\varphi}(t)$ or $g=x_{\varphi}(a)h_{\varphi}(t)s_{\varphi}$ ($a\in\K$, $t\in\K^{\times}$). Consider $g=x_{\varphi}(a)h_{\varphi}(t)$. If $a=0$ then $g=g_1$, so suppose that $a\neq 0$. Since $\langle\varphi,\alpha_8\rangle=1$ it follows from~(\ref{eq:somerelations}) that 
$$
g'=h_{\alpha_8}(a)^{-1}gh_{\alpha_8}(a)=x_{\varphi}(1)h_{\varphi}(t).
$$
If $t=1$ then $g'=g_2$ and if $t=-1$ then $g'=g_3$. If $t\neq\pm 1$ then by~(\ref{eq:somerelations}) we compute
$$
x_{\varphi}(t^{-1}/(t-t^{-1}))g'x_{\varphi}(t^{-1}/(t-t^{-1}))^{-1}=h_{\varphi}(t)=g_1.
$$
Now consider $g=x_{\varphi}(a)h_{\varphi}(t)s_{\varphi}$. We have
$
h_{\alpha_8}(t)^{-1}gh_{\alpha_8}(t)=x_{\varphi}(at^{-1})s_{\varphi},
$
and since $s_{\varphi}=x_{\varphi}(1)x_{-\varphi}(-1)x_{\varphi}(1)$ the element $g$ is conjugate to
$
x_{-\varphi}(-1)x_{\varphi}(b)
$
with $b\in\K$. If $b=0$ then the element is conjugate to $g_2$, and if $b\neq 0$ then conjugating by $s_{\varphi}h_{\alpha_8}(b^{-1})$ gives the form~$g_4$, completing the proof of the claim. 

We now show that the elements $g_1,g_2,g_3,g_4$ act domestically. The element $g_1$ is domestic with diagram $\sE_{8;0}$ (if $t=1$) or $\sE_{8;4}$ (if $t\neq 1$) by \cite[Theorem~4.7]{PVMexc}, the element $g_2$ is domestic with diagram $\sE_{8;1}$ by \cite[Theorem~2.1]{PVMexc}, and the element $g_3$ is domestic with diagram $\sE_{8;4}$ by \cite[Proposition~8.11]{npvv} (unless $\mathsf{char}(\K)=2$ in which case $g_3=g_2$ has diagram $\sE_{8;1}$). Now, $g_4$ is a product of two long root elations, each with displacement $\ell(s_{\varphi})$ (by \cite[Theorem~2.1]{PVMexc}, assuming $a\neq 0$). Hence
$
\mathsf{disp}(g_4)\leq 2\ell(s_{\varphi}).
$
We have $s_{\varphi}=w_{\sE_7}w_0$ and so $\mathsf{disp}(g_4)\leq 114<\ell(w_0)$, showing that $g_4$ is domestic.

The elements~$g_1,g_2$ and $g_3$ in the proof of Theorem~\ref{thm:equatordomestic} fix the base chamber $B$ of $\Delta=G_0/B$. Consider $g_4=x_{\varphi}(a)x_{-\varphi}(1)$. This element is conjugate to the element $\theta'=x_{-\alpha_8}(a)x_{\alpha_8}(1)$, which stabilises the residue $B\cup Bs_8B$. By \cite[Lemma~8.2]{PVMexc} the automorphism $\theta'$ (and hence $g_4$) fixes a chamber of $\Delta$ if and only if $\theta'$ fixes a chamber of the residue $B\cup Bs_8B$. If $a=0$ then $\theta'$ fixes~$B$. If $a\neq 0$ then $\theta'$ does not fix~$B$, and the remaining chambers of the residue $B\cup Bs_8B$ are of the form $x_{\alpha_8}(z)s_8B$ with $z\in\K$. We compute 
\begin{align*}
\theta' x_{\alpha_8}(z)s_8B&=\begin{cases}
B&\text{if $az+a+1=0$}\\
x_{\alpha_8}((z+1)/(az+a+1))s_8B&\text{if $az+a+1\neq 0$}
\end{cases}
\end{align*}
It follows that $\theta'$ (and hence $g_4$) fixes a chamber of $\Delta$ if and only if $aX^2+aX-1$ has a root~$z\in \K$. If $aX^2+aX-1$ is irreducible, then since $g_4$ does not fix a chamber the classifications of automorphisms with diagram $\sE_{8;1}$ and $\sE_{8;2}$ from \cite[Theorems~2.4 and~5.1]{PVMexc} imply that $g_4$ necessarily has diagram~$\sE_{8;4}$. 
\end{proof}

\begin{rem}\label{rem:highrank}
The technique used in Theorem~\ref{thm:equatordomestic} to prove that elements of the form $\theta=x_{\varphi}(a)x_{-\varphi}(1)$ (with $a\neq 0$) are domestic extends to all types of sufficiently high rank, since it shows that the numerical displacement of $\theta$ is bounded by $2\ell(s_{\varphi})$ which is linear in the rank, while $\ell(w_0)$ is quadratic in the rank. For low rank the element $\theta$ is typically not domestic. Indeed, for the simply laced diagrams, $\theta$ is domestic for $\sA_n$ with $n\geq 5$, $\sD_n$ with $n\geq 6$, and $\sE_n$ with $n\geq 7$ (and non-domestic otherwise). The statements for $\sA_n$ and $\sD_n$ are easily checked, and the fact that $\theta$ is not domestic for $\sE_6$ follows from the classification of domestic automorphisms. In type $\mathsf{E}_7$ the crude estimate above is not sufficient to prove that $\theta$ is domestic, however the element $\theta$ is indeed domestic (see the third paragraph of \cite[Section 6.1]{npvv}, where a geometric proof is given using the minuscule geometry $\mathsf{E}_{7,7}(\mathbb{K})$ -- such a proof is not possible in the $\sE_8$ case because there is no minuscule geometry).   
\end{rem}

\begin{rem}
The elements  $\theta=x_{\varphi}(1)h_{\varphi}(-1)$ (with $\mathsf{char}(\mathbb{K})\neq 2$), $\theta=h_{\varphi}(c)$ (with $c\neq 0,1,-1$), and $\theta=h_{\varphi}(-1)$ (with $\mathsf{char}(\mathbb{K})\neq 2$) appearing in Theorem~\ref{fixedchambers} also fix an equator geometry in $\mathsf{E}_{8,8}(\mathbb{K})$. These elements fix $1$ point, $2$ points, and all points, respectively, of the imaginary line associated to the equator geometry. 
\end{rem}

\section{Proof of Theorem~\ref{nofixedchamber}}\label{sec:class}

So far we have proved that automorphisms of Class I and Class II provide examples of domestic automorphisms of the building $\sE_8(\K)$ fixing no chamber, and hence necessarily have opposition diagram $\sE_{8;4}$. We now show that these are the only examples. 
We work in the non-strong parapolar space and also long root subgroup geometry $\Delta:=\mathsf{E_{8,8}}(\K)$. 
%
%

Then we can state our main achievement of this section as:

\begin{theorem}\label{classification}
A domestic collineation of $\Delta$ fixing no chamber is either of Class \emph{I} or~\emph{II}.
\end{theorem}

The proof of Theorem~\ref{classification} is achieved in the following subsections. 
  
\subsection{A lemma and a case distinction}

Recall our assumption $|\K|\geq 3$. The following lemma is about projective planes. A \emph{projectivity} between two projective spaces defined over the same field $\K$ is a collineation between these spaces that corresponds to a $\K$-linear bijective map between the respective underlying $\K$-vector spaces. 

\begin{lemma}\label{projplane}
Let $\pi$ and $\pi'$ be two projective planes over $\K$ containing the respective lines $L$ and $L'$, and respective points $p\notin L$ and $p'\notin L'$. Let $\varphi:L'\to L$ be a given projectivity. Let $\theta:\pi\to\pi'$ be a projectivity mapping $p$ to $p'$.  Set $\theta(L)=M'$, suppose $L'\neq M'$. For each point $q\in \pi\setminus(L\cup\theta^{-1}(L'))$ define the map $\varphi_q:L\to L: z\mapsto \varphi(L'\cap\theta(qz))$. Then the following hold.
\begin{compactenum}[$(i)$]
\item If $\varphi_p$ is the identity, then there exists $q\in \pi\setminus(L\cup\theta^{-1}(L'))$ such that $\varphi_q$ has exactly two fixed points. 
\item If $\varphi_p$ is the identity, then there exists $q\in \pi\setminus(L\cup\theta^{-1}(L'))$ such that $\varphi_q$ has exactly one fixed point. 
\item If $\varphi_p$ has no fixed points, then there exists $q\in \pi\setminus(L\cup\theta^{-1}(L'))$ such that $\varphi_q$ is not the identity and has at least one fixed point. 
\item If $\varphi(L'\cap M')$ is the unique fixed point of $\varphi_p$, then there exists $q\in \pi\setminus(L\cup\theta^{-1}(L'))$ such that $\varphi_q$ has exactly two fixed points. 
\item If $\varphi_p$ has exactly two fixed points among which $\varphi(L'\cap M')$, then there exists $q\in \pi\setminus(L\cup\theta^{-1}(L'))$ such that $\varphi_q$ has exactly one fixed point. 
\end{compactenum}
\end{lemma}
\begin{proof} Set $\{x'\}=L'\cap M'$ and $x=\varphi(x')$ and note that, if $\varphi_p$ fixes $x$, then $\{x\}=L\cap\theta^{-1}(L')$.
\begin{compactenum}[$(i)$]
\item Suppose that $\varphi_p$ is the identity. Choose arbitrarily $q$ in $\pi\setminus(L\cup px\cup \theta^{-1}(L'))$. Note that this is possible since $|\K|\geq3$. Then it is easy to verify that $\varphi_q$ fixes $x$ and $qx\cap L$, but no other point of $L$. Hence $\varphi_q$ has exactly two fixed points. 
\item Suppose again that $\varphi_p$ is the identity. Choose arbitrarily $q$ in $px\setminus\{p,x\}$. Then it is easy to verify that $\varphi_q$ fixes $x$, but no other point of $L$. Hence $\varphi_q$ has exactly one fixed point. 
\item Suppose that $\varphi_p$ has no fixed points. Select $z\in L\setminus\{x\}$ with $\varphi_p(z)\neq x$. Let $q\in\pi$ be such that $\{\theta(q)\}=p'x'\cap \varphi^{-1}(z)\theta(z)$. Then $\varphi_q(x)\neq x$ and $\varphi_q(z)=\varphi(L'\cap\theta(qz))=\varphi(L'\cap \theta(q)\theta(z))=\varphi(L'\cap \varphi^{-1}(z)\theta(z))=\varphi(\varphi^{-1}(z)=z$.  
\item Select $z'_1$ and $z'_2\neq z'_1$ in $L'\setminus\{x'\}$. Let $q\in\pi$ be such that $\theta(q)$ is the intersection of the line $p'z_1'$ with the line joining $z_2'$ with $\theta(\varphi(z_2'))$. Then one easily checks that $\varphi_q$ fixes $x$ and $\varphi(z_2')$, but not $\varphi(z_1')$. Since $\varphi_q$ stems from a linear map, it fixes exactly two points. 
\item Define the following projectivity $\rho:L'\to M':z\mapsto \theta(\varphi(z))$. Then $x'$ is fixed and hence $\rho$ is a perspectivity with centre $c'$ on the line $K'$, where $\theta^{-1}(K')\cap L$ is the second fixed point of $\varphi_p$. Since the latter is not trivial, the point $c'$ does not coincide with $p'$. Now taking for $q$ the inverse image under $\theta$ of any point of $x'z'\setminus\{x',z'\}$, it follows that $\theta_q$ only fixes $x$, as is straightforward to verify.   
\end{compactenum}
This completes the proof of the lemma. 
\end{proof}

We return to proving Theorem~\ref{classification}. So let $\theta$ be an automorphism of $\mathsf{E_8}(\K)$ with opposition diagram $\mathsf{E_{8;4}}$ fixing no chamber. Select a chamber $C$ of $\mathsf{E_8}(\K)$ realising the opposition diagram, that is, such that some simplex of $C$ (of type $\{1,6,7,8\}$) of size $4$ is mapped onto an opposite simplex (such a chamber exists by cappedness, as $|\K|\geq 3$). We now view $\theta$ as a collineation of $\Delta=\mathsf{E_{8,8}}(\K)$. Then such a simplex consists of a point $p$, a line $L$, a plane $\pi$ and a symp $\xi$. Let $\theta_p$ be the collineation of $\Res_\Delta(p)$ mapping an arbitrary simplex in $\Res_\Delta(p)$ to the projection onto $p$ of its image under $\theta$. It follows from \cite[Proposition 3.29]{Tits:74} that $\theta_p$ has opposition diagram $\mathsf{E_{7;3}}$. So by \cite{npvv}, there are three possibilities.  
\begin{compactenum}[(1)] \item Either $\theta_p$ does not fix a chamber---and then its fixed structure is an isometrically embedded subgeometry $\Gamma'$ isomorphic to $\mathsf{F_{4,4}}(\K,\L)$, with $\L$ a quadratic extension of $\K$, and the points of $\Gamma'$ are lines of $\Res_\Delta(p)\cong\mathsf{E_{7,7}}(\K)$, \item or $\theta_p$  does fix a chamber and it is a generalised homology with fixed structure a non-thick building with thick frame of type $\mathsf{E_6}$, \item or it is a certain product of three perpendicular (commuting) axial elations of $\mathsf{E_{7,7}}(\K)$ with axes three symps pairwise intersecting in a line, and globally intersecting in a point. \end{compactenum} 

Regarding the third possibility, we will need the following result. 

\begin{lemma}\label{3opposite}
Let $\theta_p$ be the product of three perpendicular commuting axial elations of $\mathsf{E_{7,7}}(\K)$ with axes three symps pairwise intersecting in a line, and with global intersection a point $p$ and opposition diagram $\mathsf{E}_{7;3}$. Then every point opposite $p$ is mapped to an opposite point by~$\theta$. 
\end{lemma}

\begin{proof}
By \cite[Theorem 8.5]{npvv} $\theta_p$ is conjugate, in the $\sE_7$ Chevalley group, to an element of the form $\theta_p=x_{\varphi_1}(1)x_{\varphi_2}(1)x_{\varphi_3}(1)$ (in Chevalley generators as in Section~\ref{fixingchamber}, where $\varphi_1$ is the highest root, $\varphi_2$ is the highest root of the $\mathsf{D}_6$ diagram, and $\varphi_3=\alpha_7$). The type $7$ points opposite the fixed base point are identified with the $P_7$ cosets in $P_7w_0P_7$ (where $P_7$ is the standard parabolic subgroup of type $\mathsf{E}_6$). For such a point $gw_0P_7$ (with $g\in P_7$) we have
\begin{align*}
w_0^{-1}g^{-1}\theta_p gw_0&\in P_7x_{-\varphi_1}(1)x_{-\varphi_2}(1)x_{-\varphi_3}(1)P_7\\
&\subseteq P_7U_{\varphi_1}s_{\varphi_1}U_{\varphi_1}U_{\varphi_2}s_{\varphi_2}U_{\varphi_2}U_{\varphi_3}s_{\varphi_3}U_{\varphi_3}P_7\\&\subseteq P_7s_{\varphi_1}s_{\varphi_2}s_{\varphi_3}P_7,
\end{align*}
where we use the fact that the root elations commute (see \cite[Theorem~8.1]{npvv} for similar calculations). By \cite[Lemma~3.5]{PVMexc} we have $s_{\varphi_1}s_{\varphi_2}s_{\varphi_3}=w_{\mathsf{E}_6}w_0$, and hence $w_0g^{-1}\theta_p gw_0\in P_7w_0P_7$, and so $\theta_p(gw_0P_7)$ is opposite $gw_0P_7$. 
\end{proof}

We treat the three cases (1), (2) and (3) separately below. 

Note that, similarly to these cases, if $q$ is a non-domestic point of $\Delta$ for $\theta$, and there does not exist a non-domestic simplex of size $4$ containing $q$, then $\theta_q$ is either trivial, or it has opposition diagram either $\mathsf{E_{7;1}}$ (and $\theta_q$ is an axial elation in $\Res_\Delta(q)$) or $\mathsf{E_{7;2}}$ (and $\theta_q$ is the product of two perpendicular such elations). We will use this fact below. 

A non-domestic point $q$ will be said \emph{to have Type $1$ or $2$} if it is contained in non-domestic simplex of size $4$ and case (1) or (2), respectively, above applies. It \emph{has Type $3$} in all other cases. 

In the sequel, the \emph{upper residue} of a plane $\pi$ of $\Delta$ is the residue of that plane in the parapolar space, that is, the point residual at $\pi$ of the point residual at $L$ of the point residual at $x\in\pi$, where $L$ is a line in $\pi$ through $x$. This conforms to the star of the simplex $\{x,L,\pi\}$ in the building. Similarly, one defines the upper residue of a line. We now have the following useful result. 

\begin{lemma}\label{types}
Suppose $\alpha$ is a plane of $\Delta$ through $p$ and fixed by $\theta_p$. Then, if $p$ has Type~$1$, then so has every non-domestic point of $\alpha$. If, moreover, $\theta_p$ does not fix the upper residue of $\alpha$ elementwise, then each non-domestic point of $\alpha$ has the same type (with the above definition of type) as $p$. 
\end{lemma}
\begin{proof}
Note that $\theta_p$ and $\theta_q$ coincide over the upper residue of $\alpha$ because the projection between points is type preserving.

If $p$ has Type $1$, then in the upper residue of $\alpha$, there are no fixed chambers, whereas this is not the case in any of the other types, because there is always some fixed chamber $C$ in $\Res_\Delta(p)$ , and so the (building theoretic) projection of $C$ onto $\pi$ yields a fixed chamber containing $\pi$.  Hence $q$ also has Type 1.

Now suppose that $\theta_p$ does not fix the upper residue of $\alpha$ elementwise. If $p$ has Type $2$, then in the upper residue of $\alpha$, there is a nontrivial homology induced, whereas in Type $3$, this is a unipotent element. Now, a unipotent element is only a homology if it is the identity. The lemma follows.
\end{proof}

\subsection{Type $1$: $\theta_p$ does not fix a chamber}\label{type1}
Since in this case the fixed structure of $\theta_p$ is an isometrically embedded subgeometry $\Gamma'$ isomorphic to $\mathsf{F_{4,4}}(\K,\L)$, with $\L$ a quadratic extension of $\K$, and the points of $\Gamma'$ are lines of $\Res_\Delta(p)\cong\mathsf{E_{7,7}}(\K)$, there exists a fixed plane $\alpha$ of $\theta_p$. Set $L=\alpha\cap (p^\theta)^{\Join}$ and $L'=\alpha^\theta\cap p^{\Join}$. Set $M'=L^\theta$. Assume for a contradiction that $L'\neq M'$. Then, since $\Gamma'$ does not contain lines through $p$, the projectivity $L\to L$ defined by $x\mapsto L\cap((px)^\theta\cap L')^\perp$ does not have fixed points. However, by \cref{projplane}$(iii)$, setting $\varphi:L'\to L:x\mapsto L\cap x^\perp$, we can choose $q\in\alpha$ such that the analogously defined map does have fixed points, contradicting the fact that $q$ also has Type 1 by \cref{types}.  We have shown that $L^\theta=\alpha^\theta\cap p^{\Join}$.

Now let $\zeta$ be a symp through $p$ fixed by $\theta_p$. Then $\zeta\cap\zeta^\theta$ is a point $u$. Also, each line $M$ of $\zeta$ through $p$ is contained in a fixed plane, so the previous paragraph implies that $(M\cap (p^\theta)^{\Join})^\theta=M^\theta\cap p^{\Join}$. It follows that $p^\perp\cap (p^\theta)^{\Join}\cap\zeta$, which equals $p^\perp\cap u^\perp$, is mapped onto $(p^\theta)^\perp\cap p^{\Join}\cap \zeta^\theta$, which equals $(p^\theta)^\perp\cap u^\perp$. But since $\zeta$ is hyperbolic, $u$ and $p$ are the only points of $\zeta$ collinear to $p^\perp\cap u^\perp$; likewise $u$ and $p^\theta$ are the only points collinear to $(p^\theta)^\perp\cap u^\perp$. This implies $\{u,p\}^\theta=\{u,p^\theta\}$, yielding $u^\theta=u$. We have shown:

\begin{lemma}\label{inE}
With the above notation and conventions, if $\theta_p$ fixes a symp $\xi$, then it fixes the unique point of $\xi$ symplectic to $p^\theta$.  
It follows that the equator geometry $E(p,p^\theta)=p^{\pperp}\cap(p^\theta)^{\pperp}$ contains a fully isometrically embedded pointwise fixed subgeometry $\Gamma''\cong\mathsf{F_{4,1}}(\K,\L)$.  
\end{lemma}

Now consider two opposite points $x,y$ of $\Gamma''$. Then the collineation $\theta'$ induced in $E(x,y)$ by $\theta$ is domestic (as every singular subspace of $\Delta$ is domestic) and does not fix any chamber (as this would induce a fixed chamber in $\Delta$ containing $x$). By \cite{npvv}, there are three (nested) possibilities.

\begin{itemize}
\item \emph{The opposition diagram of $\theta'$ is $\mathsf{E_{7;3}}$}. In this case, the opposition diagram of $\theta'_p$ (which is the restriction of $\theta_p$ to $E(x,y)$) is one of $\mathsf{D_{6;0}}$, $\mathsf{D_{6;1}^1}$, $\mathsf{D_{6;1}^2}$ or $\mathsf{D_{6;2}^1}$ (because these are the opposition diagrams of which the encircled nodes are a subset of those encircled in the opposition diagram $\mathsf{D_{6;2}^1}$ obtained from  $\mathsf{E_{7;3}}$ by taking the residue of a node of type $1$). In all these cases \cite[Theorem 1]{PVMclass} implies that $\theta'_p$ pointwise  fixes planes of the corresponding polar space. Let $\pi$ be such a plane. This conforms to a maximal singular $3$-space of the associated geometry of type $\mathsf{D_{6,6}}$, hence to a maximal singular $4$-space $U$ containing $p$ of $E(x,y)$, and $\theta'_p$ fixes all symps and paras containing $U$.  The latter relates to a vertex of the branching type in $\Res_\Delta(p)$, hence to a plane $\pi$ of $E(p,p^\theta)$. Since all symps throught $U$ are fixed, \cref{inE} implies that $\pi$ is a pointwise fixed plane in $E(p,q)\cap E(x,y)$, contradicting the fact that the fixed points in the latter form the equator in an $\mathsf{F_{4,1}}(\K,\L)$, which does not contain lines (see for instance  \cite{Pet-Mal:23}). 
\item \emph{The opposition diagram of $\theta'$ is $\mathsf{E_{7;4}}$}. Since no chamber is fixed, \cite{npvv} implies that this case splits in two possibilities.
\begin{itemize} \item $\theta'$ pointwise fixes an equator geometry of $E(x,y)$ (as geometry isomorphic to $\mathsf{E_{7,1}}(\K)$), say $E(a,b)\cap E(x,y)$, with $a$ and $b$ two opposite points of $E(x,y)$. Select a point $z$ in $\Gamma''$ symplectic to both $x$ and $y$. Then $z\in E(a,b)$ as the only fixed points of $E(x,y)$ are those of $E(a,b)$. Set $\xi=\xi(x,z)$. Set $\zeta=\xi\cap E(a,b)$ and let $\zeta'$ be the intersection of $\xi$ with the point set of $\Gamma''$.

First we note that, in the geometry $E(a,b)$, two opposite points $x,y$ and their equator are pointwise fixed. Hence all symps of $E(a,b)$ through $x$ or $y$ are stabilised, which implies that the residue of $x$ is pointwise fixed. This, in turn,  implies that $\zeta\cap (z^\perp\cap x^\perp)\subseteq E(a,b)\cap (y^{\Join}\cap x^\perp)$ is pointwise fixed.  

 We now argue in the ambient projective space $W$ of $\xi$, which is allowed as $\theta$ stabilised $\xi$ and extends uniquely to a collineation in $W$. Note that $\xi\cong\mathsf{D_{7,1}}(\K)$, $\zeta\cong\mathsf{D_{5,1}}(\K)$ and $\zeta'\cong\mathsf{B_{3,1}}(\K,\L)$. Hence, in $W$, the dimensions of $\<\zeta\>$ and $\<\zeta'\>$ are $9$ and $7$, respectively. Consequently, $\<\zeta\>\cap\<\zeta'\>$ is at least $3$-dimensional and contains the line $\<x,z\>$. Since $\theta$ pointwise fixes $\zeta'$, it pointwise fixes $\<\zeta\>\cap\<\zeta'\>$. Since it also pointwise fixes $\<\zeta\cap z^\perp\cap x^\perp\>$, which is complementary to $\<x,z\>$ in $\<\zeta\>$, we deduce that $\theta$ pointwise fixes $\zeta$. 
 
 Now let $\widetilde{\zeta}$ be a symp of $E(a,b)$ containing $x$ and adjacent to $\zeta$, that is, intersecting $\zeta$ in a maximal singular subspace (here $4$-dimensional) $U$. Set $\{\widetilde{z}\}=\widetilde{\zeta}\cap E(x,y)$. Then, as above, $\theta$ pointwise fixes $\widetilde{\zeta}\cap\widetilde{z}^\perp\cap x^\perp$. Together with the fact that $U$ is pointwise fixed, this implies that $\theta$ pointwise fixes the tangent hyperplane $H_x$ at $x$ of $\widetilde{\zeta}$, and also $\widetilde{z}$. Hence $\theta$ induces in $\<\widetilde{q}\>$ a homology with axis $H_x$ and centre $\widetilde{z}$. But since there are lines through $\widetilde{z}$ intersecting $\widetilde{\zeta}$ in exactly one other point, this has to be the identity. We conclude that $\theta$ pointwise fixes $\widetilde{\zeta}$. Going on like this, we find that $\theta$ fixes each point of $E(a,b)\cap (x^\perp\cup x^{\pperp})$.  Then \cref{elationsE7} implies that $\theta$ induces a central elation in $E(a,b)$ with centre $x$, which fixes the point $y$ opposite $x$. We conclude that $\theta$ pointwise fixes $E(a,b)$
 and $\theta$ is domestic of Class II. 
 
 \item $\theta'$ pointwise fixes a fully and isometrically embedded polar space $\Delta'''$ isomorphic to $\mathsf{C_{3,1}}(\H,\K)$, for some quaternion division algebra over $\K$. Note that `isometric' here means that collinear points of  $\mathsf{C_{3,1}}(\H,\K)$ are symplectic in $\Delta$, and opposite points of $\mathsf{C_{3,1}}(\H,\K)$ are also opposite in $\Delta$.  The geometry $\Delta'''$ isometrically contains the subgeometry $\Gamma'''$, induced by $E(x,y)\cap E(p,p^\theta)$, and isomorphic to $\mathsf{C_{3,1}}(\L,\K)$. This shows $\K\leq\L\leq\H$. 

Now consider any line $K$ through $x$ fixed by $\theta$. Then the corresponding para in $E(x,y)$ is fixed, and hence belongs to $\Delta'''$. Since the paras of $\Delta'''$ correspond to planes of the associated polar space $\mathsf{C_{3,1}}(\H,\K)$, we find that some plane of the subgeometry $\mathsf{C_{3,1}}(\L,\K)$ corresponding to $\Gamma'''$ is opposite the plane corresponding to $K$. Hence the corresponding line $K'$ through $x$ is opposite $K$ in $\Res_\Delta(x)$, and is pointwise fixed (since it is contained in $\Gamma''$), just like its projection $K''$ onto $y$. But now $K''$ is opposite $K$ (use \cite[Proposition~3.29]{Tits:74}). It follows that $\theta$ pointwise fixes $K$. The fixed lines through $x$ in a fixed symp correspond to the fixed paras in $E(x,y)$ through a fixed point, and hence constitute a geometry isomorphic to $\mathsf{B_{2,1}}(\K,\H)$. Hence the fixed point structure of $\theta$ induced in any stabilised symp through $x$ is a geometry isomorphic to $\mathsf{B_{3,1}}(\K,\H)$. 

We now claim that every fixed point $z$ is opposite some point of $\Gamma''$. Indeed, if $z$ is collinear to $x$, we exhibited an opposite fixed point collinear to $y$ above; if $z$ is special to $x$, then a similar argument exhibits an opposite fixed point collinear to $x$ using   \cref{spspop}. Since $x$ was chosen arbitrarily, we may assume that $z$ is symplectic to all points of $\Gamma''$. Hence $z$ belongs to $\Delta'''$, which clearly contains points opposite any of its own points. The claim is proved.  Hence the fixed point structure in the residue of each fixed point is the same as the one in $x$, and is isomorphic to $\mathsf{C_{3,1}}(\H,\K)$. 

Next we claim that also the fixed point structure in each fixed symp $\xi^*$ is the same, and isomorphic to $\mathsf{B_{3,1}}(\K,\H)$.   Indeed, clearly not all fixed points are close to $\xi^*$, hence at least one is far and induces a fixed point $z$ in $\xi^*$. We may treat $z$ as $x$ above and the claim follows.  Let $\Gamma$ be the geometry of fixed points and (pointwise) fixed lines. 

The previous paragraph now implies that for two fixed symplectic points $a,b$, the convex closure in $\Gamma$ is a subquadric of $\xi(a,b)$ isomorphic to $\mathsf{B_{3,1}}(\K,\H)$. Hence $\Gamma$ is a parapolar space all point residuals are isomorphic to  $\mathsf{C_{3,1}}(\H,\K)$. Since we can now derive that the fix diagram is $\mathsf{E_{8;4}}$, we know that $\Gamma$ is associated to a building of $\mathsf{F_4}$, and we hence conclude that $\Gamma$ is a metasymplectic space isomorphic to $\mathsf{F_{4,1}}(\K,\H)$, isometrically and fully embedded in $\Delta$. Hence $\theta$ is of Class I.
\end{itemize}
\end{itemize}   
\begin{rem}Note that, if $\theta$ is of Class I, pointwise fixing a metasymplectic space $\Gamma$, then $\theta$ does not fix any vertex of the associated building apart from those of $\Gamma$. Indeed, such vertex corresponds to a singular subspace $U$(because we already proved that every stabilised symp belongs to $\Gamma$). It is easy to see that there exists a fixed point $f$ opposite some point of $U$. Then also the projection $U'$ of $x$ onto $U$ is stabilised, and this is a singular subspace of dimension one less than the dimension of $U$. Repeating this argument a few times leads to a fixed point in $U$, and our proof above implies that $U$ belongs to $\Gamma$.
\end{rem}

\subsection{Type $2$: $\theta_p$ is a generalised homology}\label{type2}
By \cite{PVMclass}, $\theta_p$ pointwise fixes a non-thick building with thick frame of type $\mathsf{E_6}$; in Chevalley notation it is the homology $h_{\omega_7}(c)$, with $c\in\K\setminus\{0,1\}$, see \cite[Theorem~2$(ii)$]{npvv}.  Hence, according to the last paragraph of \cref{equatorgeom}, we can describe the fixed point set of $\theta_p$, acting on the symps through $p$, hence acting on $\mathsf{E_{7,1}}(\K)$, as the point sets of two opposite paras $\Pi$ and $\Pi'$ (isomorphic to $\mathsf{E_{6,1}}(\K)$, together with the point set of their equator geometry $E(\Pi,\Pi')$. As point set of $\Res_\Delta(p)$, it is the union of two opposite points and their traces. Let $\pi$ be a plane through $p$ that corresponds to a line in $\Res_\Delta(p)$ intersecting both the mentioned traces.   
Then $\pi$ is fixed by $\theta_p$, but the upper residue of $\pi$ is not fixed pointwise, because $\pi$ corresponds to a symp in $\mathsf{E_{7,1}}(\K)$ containing lines that are not contained in $\Pi\cup\Pi'\cup E(\Pi,\Pi')$. 
We can now use the same technique as in the first paragraph of \cref{type1}, using \cref{projplane}$(ii)$ and \cref{types} to conclude that $\pi\cap (p^\theta)^{\Join}$ is mapped onto $\pi^\theta\cap p^{\Join}$, and consequently that the unique point symplectic to $p^\theta$ of each symp through $\pi$ fixed by $\theta_p$ is fixed by $\theta$. But every symp through $p$ fixed by $\theta_p$ contains such a plane $\pi$, so we conclude that 
$E(p,p^\theta)$ contains a fixed point set $\Gamma''$ isomorphic to the long root geometry of a non-thick building of type $\mathsf{E_7}$ having thick frame type $\mathsf{E_6}$, in particular the union of two paras $\Pi$ and $\Pi'$ and their equator geometry $E(\Pi,\Pi')$. 

Consider two opposite points $x,y$ of $\Gamma''$, more exactly in $E(\Pi,\Pi')$. The equator geometry $E(x,y)$ (viewed inside $\Delta$) is stabilised by $\theta$ and hence, as before, $\theta$ induces in $E(x,y)$ a domestic  collineation not fixing any chamber. Recall from \cite[Theorem~1]{npvv} that there are three possibilities for the fixed point sets of such collineations, but all of them are subspaces, that is, if two collinear points are fixed, then the joining line is fixed pointwise. However, if we restrict $E(x,y)$ to $E(p,p^\theta)$, then, by \cite[Proposition~6.15]{DSV2}, we obtain a stabilised geometry $\mathsf{D_{6,2}}(\K)$. The intersection of $\mathsf{D_{6,2}}(\K)$ with $\Gamma''$ is the restriction of $\Pi\cup\Pi'\cup E(\Pi,\Pi')$ to $E(x,y)$, which is readily seen to contain points of $\Pi$ and points of $E(\Pi,\Pi')$ which are collinear. But the line through these points is not fixed pointwise, and that is a contradiction. (More exactly, the fixed point structure consists of two opposite paras $\Upsilon$ and $\Upsilon'$ of $E(x,y)$ isomorphic to $\mathsf{A_{5,2}}(\K)$ together with their equator geometry $E(\Upsilon,\Upsilon')$ (defined similarly as $E(\Pi,\Pi')$), and this is not a subspace.)

Hence this case does not lead to an example.

\subsection{Type $3$: $\theta_p$ is the product of three pairwise perpendicular axial elations}\label{type3}
Since \begin{compactenum}[$\bullet$]\item this is the only case remaining, and \item using \cite[Proposition 3.29]{Tits:74}, the opposition diagram of $\theta_q$, for any non-domestic point $q$, is one of $\mathsf{E_{7;0}}$, $\mathsf{E_{7;1}}$, $\mathsf{E_{7;2}}$ or $\mathsf{E_{7;3}}$, and \item each collineation of $\mathsf{E_{7,1}}(\K)$ with opposition diagram $\mathsf{E_{7;1}}$ or $\mathsf{E_{7;2}}$ is the product of one or two (perpendicular) central elations, \end{compactenum}we may assume that 
\begin{quote} (*) \  \ \ \ \ \ \ \ \emph{for any non-domestic point $q$, the collineation $\theta_q$ is unipotent. } \end{quote}

This implies for instance, that each fixed line either has exactly one fixed point (and then we call the action of $\theta_q$ on $L$ an \emph{elation}), or is pointwise fixed. 

We first prove the following claim about hyperbolic polar spaces.

\begin{lemma}\label{hypel}
Let $\varphi$ be a collineation of $\mathsf{D}_{n,1}(\K)$ fixing some point $x$ and all lines through it. If $\varphi$ is unipotent, then it is the product of two axial elations. Their axes are not coplanar if, and only if, $\varphi$ is not an axial elation itself if, and only if, $\varphi$ maps some point to a non-collinear point. 
\end{lemma}

\begin{proof}
We may represent $\mathsf{D}_{n,1}(\K)$ as a hyperbolic quadric $Q$ in $\PG(2n-1,\K)$ and choose coordinates $(x_{-n},\ldots,x_{-1}.x_1,\ldots,x_n)$ in such a way that $Q$ has equation \[x_{-n}x_n+x_{-n+1}x_{n-1}+\cdots+x_{-2}+x_2+x_{-1}+x_1=0.\] One verifies easily that a generic unipotent collineation of $Q$ fixing all lines through the point $(1,0,0,\ldots,0)$ stems from a linear map with matrix
\[M_{A,B}:=\begin{pmatrix} 1 & A & A\cdot B & B\\ 0 & I_{n-1} & -B^{\mathsf{T}} & 0 \\ 0 & 0 & 1 & 0\\ 0 & 0 & -A^{\mathsf{T}} & I_{n-1}\end{pmatrix},\] where $A$ and $B$ are two arbitrary $(n-1)$-tuples with entries in $\K$, and $A\cdot B$ is the usual scalar product. One now computes that $M_{A,B}=M_{A,0}M_{0,B}$. The matrix $M_{A,0}$ represents an axial elation with axis generated by the points $(1,0,\ldots,0)$ and $(1,0,\ldots,0,A)$, whereas $M_{0,B}$ represents an axial elation with axis generated by $(1,0,\ldots,0)$ and $(1,B,0,\ldots,0)$. One verifies that $\varphi$ is an axial elation itself if, and only if, $A\cdot B=0$. The latter is also equivalent to  $(1,0,\ldots,0,A)$ and $(1,B,0,\ldots,0)$ being collinear. The last equivalence follows from noting that an axial collineation is point-domestic and the image of the point whose coordinates are all $0$ except for $x_1=1$ is $(A\cdot B,-B,1,A)$. 
\end{proof}

Now the fixed point set of $\theta_p$, viewed as a collineation in $\mathsf{E_{7,1}}(\K)$, is a para $\Pi$ (all points symplectic to three given pairwise symplectic points not contained in a common symp; this para corresponds to the global intersection point of the three axes when $\theta$ is viewed as a collineation of $\mathsf{E_{7,7}}(\K)$). To $\Pi$ corresponds a line $L$ of $\Delta$ through $p$, which on its turn corresponds to a point $p_L$ of $\mathsf{E_{7,7}}(\K)\cong\Res_\Delta(p)$. Each symp $\xi_L$ through $p_L$ is stabilised and, by \cref{hypel}, the restriction of $\theta_p$ to such symp is a product of two axial elations the axes $A_L$ and $B_L$ of which are two non-coplanar lines through $p_L$. Hence all lines through $p_L$ are stabilised with induced action either the identity or a translation.  Hence, if $\alpha$ is a plane of $\Delta$ containing $L$, then the same argument as in the first paragraph of \cref{type1}, using \cref{projplane} and (*), shows that $\theta$ maps  $\alpha\cap (p^\theta)^{\Join}$ to $\alpha^\theta\cap p^{\Join}$.

Now let $\xi_L$ be any symp through $p_L$ in $\Res_\Delta(p)$, and let $M_L$ be a line in $\xi_L$ intersecting, with above notation, the axis $A_L$ in some point $y_L$, and not collinear to $p_L$. Then, by (*), the map sending a point $x\in M_L$ to the projection of $x^{\theta_p}$ to $M_L$ is a translation with fixed point  $y_L$. Let $\alpha$ be the plane in $\Delta$ through $p$ corresponding to the line $M_L$. Set $K=\alpha\cap (p^\theta)^{\Join}$ and $K'=\alpha^\theta\cap p^{\Join}$. 
Then, by \cref{perp-perp}, the foregoing implies that the mapping $K\to K:x\mapsto (K'\cap (px)^\theta)^{\pperp}\cap K$ is a translation with (unique) fixed point the point on $K$ corresponding to $y_L$. If $K'\neq K^\theta$, then using \cref{projplane}$(iv)$, we can select a non-domestic point $q$ in $\alpha$ such that the analogously defined map for $q$ and $q^\theta$ has exactly two fixed points. However, composing with a suitable axial elation (using the other axis of $\theta_q$), and by virtue of (*), we then obtain a unipotent map in a symp fixing exactly two points on some line, a contradiction. We conclude that $K'=K^\theta$. Since $\xi_L$ is, as a subgeometry, generated by $p_L^\perp\cap \xi_L$ and $M_L$, this implies, if $\xi$ is the symp through $p$ corresponding to $\xi_L$,  that $\xi\cap (p^\theta)^{\Join}$ is mapped onto $\xi^\theta\cap p^{\Join}$. Since the projection of $\xi^\theta$ onto $p$ is exactly $\xi$, the same argument as in the second paragraph of \cref{type1} shows that $\xi\cap \xi^\theta$ is a fixed point of $\theta$. 

Now let $M_L$ be any line in $\Res_\Delta(p)$ containing points opposite $p_L$. Then, 
with the same notation for $\alpha, K$ and $K'$ as in the previous paragraph, and using \cref{3opposite}, the mapping $K\to K:x\mapsto (K'\cap (px)^\theta)^{\Join}\cap K$ is a translation.  Similarly as in the previous paragraph, this leads, using \cref{projplane}$(iv)$, to a contradiction if $K'\neq K^\theta$. Varying $M_L$, this now implies that $\Xi:=p^\perp\cap(p^\theta)^{\Join}$ is mapped onto $\Xi':=(p^\theta)^\perp\cap p^{\Join}$.  Now, since hyperbolic lines in hyperbolic quadrics have size 2, we have
\[E(p,p^\theta)=\{x\pperp p\mid x^\perp\cap p^\perp\subseteq \Xi\}=\{x\pperp p^\theta\mid x^\perp\cap (p^\theta)^\perp\subseteq \Xi'\}=E(p,p^\theta)^\theta,\] and so we conclude that $E(p,p^\theta)$ is stabilised by $\theta$. The line $L$, which is ``sympwise'' fixed, corresponds to the pointwise fixed para $\Pi$, by which we now also denote the corresponding para in $E(p,p^\theta)$.  

Since the opposition diagram of $\theta_p$ is $\mathsf{E_{7;3}}$, there exists some point $q\in E(p,p^\theta)$ mapped to an opposite point $q^\theta$. Clearly, $q$ is not close to $\Pi$. 
Then \cref{paraimline} yields a unique point $q^*\in\cI(q,q^\theta)\cap\Pi$. Now, we may assume that $q$ has Type 3, and so $\theta$ fixes a chamber in $E(q,q^\theta)=E(q,q^*)$, yielding a fixed chamber through $q^*$ in $\Delta$, a contradiction.   Hence this case does not lead to an example.

This completes the proof of \cref{classification}.

Now  \cref{nofixedchamber} follows from \cref{classIisdom}, \cref{thm:equatordomestic} and \cref{classification}.

\section{Density Theorems and further applications}\label{sec:applications}

As noted in the introduction, the present paper completes the classification of domestic automorphisms of large spherical buildings of exceptional type, of rank at least~$3$. In this section we record some applications and consequences of this classification. 

The first application, in Section~\ref{sec:density}, provides ``density theorems'' for groups $G$ acting strongly transitively on large spherical buildings of exceptional type. Specifically, we use the theory of domesticity and opposition diagrams to show that each conjugacy class in $G$ intersects a union of a very small number of $B$-cosets (with $B$ the stabiliser of a fixed choice of base chamber of the building). 

In Section~\ref{sec:kangaroo} we focus on Class I automorphisms of large $\sE_8$ buildings, and show that the displacement spectrum of points in the long root geometry $\sE_{8,8}(\K)$ of such an automorphism has a ``spectral gap'' property, meaning that certain distances in the spectrum are skipped. 

Finally, in Section~\ref{sec:biclass} we give a natural extension of the concept of uniclass automorphism, and use the classification of domestic automorphisms to classify those type preserving automorphisms of Moufang spherical buildings whose displacement spectra contains the identity and only one other conjugacy class in the Weyl group. 

\subsection{Density Theorems}\label{sec:density}

In this subsection we discuss Density Theorems for groups acting on spherical buildings (and in particular prove Corollary~\ref{cor:C}). More precisely, let $\Omega$ be a large spherical building of irreducible type $(W,S)$ of rank at least~$3$, considered as a chamber system with Weyl distance function $\Delta:\Omega\times\Omega\to W$. Let $G$ be a group of type preserving automorphisms of $\Omega$ (that is, if $\delta(C,D)=s$ then $\delta(C^{\theta},D^{\theta})=s$, for $s\in S$). Furthermore, assume that $G$ acts strongly transitively on $\Omega$ (that is, $G$ is transitive on pairs $(C,\Sigma)$ with $\Sigma$ an apartment, and $C$ a chamber of $\Sigma$). Let $C_0\in\Omega$ be a fixed choice of chamber, and let $\Sigma_0$ be a fixed choice or apartment containing~$C_0$. Let $B$ (respectively $N$) be the subgroup of $G$ fixing $C_0$ (respectively stabilising $\Sigma_0$). Then $(B,N)$ is a $BN$-pair (or Tits system) in $G$, with Coxeter system~$(W,S)$.

Let $\mathcal{O}_{\mathsf{tp}}(W,S)$ denote the set of all possible opposition diagrams of type preserving automorphisms of a large spherical building of type $(W,S)$. These diagrams are classified in \cite{PVMclass}, specifically they are the diagrams appearing in \cite[Tables~1--5]{PVMclass} with $\pi_{\theta}=\mathrm{id}$ (in the notation of \cite{PVMclass}). For each diagram $\Gamma\in \mathcal{O}_{\mathsf{tp}}(W,S)$ let $J(\Gamma)$ denote the subset of $S$ consisting of the encircled nodes, and let $w_{\Gamma}=w_{S\backslash J(\Gamma)}w_S$, where for subsets $K\subseteq S$ we write $w_K$ for the longest element of the parabolic subgroup~$W_K=\langle K\rangle$. We can now state a general Density Theorem.

\begin{theorem}\label{thm:density}
In the above setup, if $\mathcal{C}$ is a conjugacy class in $G$ then $\mathcal{C}$ intersects at least one coset $w_{\Gamma}B$ with $\Gamma\in\mathcal{O}_{\mathsf{tp}}(W,S)$. 
\end{theorem}

\begin{proof}
Let $\mathcal{C}$ be a conjugacy class in $G$, and let $\theta\in\mathcal{C}$. Then $\theta$ acts on $\Omega$ as a type preserving automorphism. Let $\Gamma$ be the opposition diagram of~$\theta$. By \cite[Theorem~2.6]{PVMclass} there is a chamber $C\in\Omega$ with $\delta(C,\theta (C))=w_{\Gamma}$. Since $G$ acts strongly transitively we have $C=gC_0$ for some $g\in G$ and hence $g^{-1}\theta g\in Bw_{\Gamma}B$. So there is $b_1,b_2\in B$ with $g^{-1}\theta g=b_1w_{\Gamma} b_2$, and hence $(gb_1)^{-1}\theta (gb_1)=w_{\Gamma}b_2b_1^{-1}\in w_{\Gamma}B$ as required. 
\end{proof}

For some nontrivial opposition diagrams $\Gamma$ it is known that every automorphism with this diagram necessarily fixes a chamber. In such an instance one may remove $\Gamma$ from the set $\mathcal{O}_{\mathsf{tp}}(W,S)$ and Theorem~\ref{thm:density} remains valid. For example, it is shown in \cite{PVMexc} that every type preserving domestic automorphism of a large building of type $\sE_6$ necessarily fixes a chamber, and therefore for type $\sE_6$ Theorem~\ref{thm:density} can be refined to the statement that every conjugacy class in $G$ intersects $B\cup w_0B$. 

Corollary~\ref{cor:C} is now immediate. 

\begin{proof}[Proof of Corollary~\ref{cor:C}] 
By \cite{PVMexc} all automorphisms with opposition diagram either $\sE_{8;1}$ or $\sE_{8;2}$ necessarily fix a chamber, and so (1) follows from Theorem~\ref{thm:density}. For (2), note that, if $\K$ is quadratically closed, then by Theorem~\ref{nofixedchamber} all automorphisms with opposition diagram $\sE_{8;4}$ also fix a chamber, and hence every domestic automorphism fixes a chamber, and the result follows. 
\end{proof}

For completeness, we also state a corresponding result for buildings of type $\sE_7$. 

\begin{cor}\label{cor:densityE7}
Let $G$ be a Chevalley group of type $\mathsf{E}_7$ over a field~$\mathbb{K}$ with $|\K|>2$, and let $\mathcal{C}$ be a conjugacy class in $G$. Then
\begin{compactenum}[$(1)$]
\item $\mathcal{C}\cap (B\cup \cup w_{\mathsf{D}_4}B\cup s_2s_5s_7w_0B\cup w_0B)\neq\emptyset$;
\item if $\mathbb{K}$ is quadratically closed then $\mathcal{C}\cap (B\cup w_0B)\neq\emptyset$. 
\end{compactenum}
Here $w_0$ is the longest element of~$W$, $w_{\sD_4}$ is the longest element of the standard $\sD_4$ parabolic subgroup, and we use Bourbaki labelling for the simple reflections. 
\end{cor}

\begin{proof}
The proof is similar to Corollary~\ref{cor:C}, as all automorphisms with opposition diagram $\sE_{7;1}$ or $\sE_{7;2}$ necessarily fix a chamber, and if $\K$ is quadratically closed then all automorphisms with opposition diagram either $\sE_{7;3}$ or $\sE_{7;4}$ also fix a chamber (by~\cite{npvv}). 
\end{proof}

\subsection{Spectral properties of Class I collineations}\label{sec:kangaroo}

Domestic collineations in buildings of type $\mathsf{E_7}$ have rather nice \emph{kangaroo} properties, that is, the displacement spectra of vertices of certain types are rather restricted (certain distances are ``skipped''). This in particular holds for the vertices corresponding to the polar node---the points of the long root subgroup geometry. Moreover, these collineations behave ``locally'' like involutions in the sense that, if some vertex $v$ is mapped onto a non-opposite vertex $w$, then any vertex incident with both $v$ and $w$ and  uniquely defined by $v$ and $w$, is fixed. We show that analogous things hold in the case of domestic collineation of $\mathsf{E_8}(\K)$ of Class I. 

Throughout, we consider a domestic collineation $\theta$ of Class I, acting on $\Delta=\mathsf{E_{8,8}}(\K)$. 

Firstly, we consider the displacement of the points. 
\begin{lemma}\label{nospecialE8}
No point of $\mathsf{E_{8,8}}(\K)$ is mapped onto a special point.
\end{lemma}

\begin{proof}
Suppose for a contradiction that the point $x$ is mapped onto the point $y$ that is special to $x$. Let $x\perp z\perp y$. Then we find a line $L$ through $x$ containing points special to both $z$ and $z^{\theta^{-1}}$ (this follows from \cite[Proposition~3.30]{Tits:74} applied to $\Res_\Delta(x)$). \cref{spspop} implies that $L$ is not domestic. The collineation $\theta_L$ of the upper residue of $L$ defined by projecting the image of an object incident with $L$ from $L^\theta$ back to $L$, is a domestic duality. By the classification of opposition diagrams of dualities in buildings of type $\mathsf{E_6}$ in \cite{PVM:19} and \cite{PVMsmall}, we find a non-domestic plane $\alpha$ through $L$. The mapping $\theta_\alpha$ on the lower residue now, and similarly defined as $\theta_L$, is a duality, for which the point $x$ is absolute, that is, $x$ is incident with its image. It follows that the line $x^{\theta_\alpha}=:K\subseteq\alpha$ through $x$ contains a point $p$ mapped to an opposite. Since $p$ is special to $y$ by definition of $K$, the lines $px$ and $p^\theta y$ are not opposite. Since $\alpha$ is opposite $\alpha^\theta$, it follows that $px$ and $(px)^{\theta_p}$ are symplectic. But using \cref{coruniclass}, this contradicts \cite[Lemma~7.15]{npvv}, which says that $\theta_p$, viewed as a collineation of $\Res_\Delta(p)$, does not map points to symplectic ones. 
\end{proof}

\begin{lemma}\label{nocollinearE8}
No point of $\mathsf{E_{8,8}}(\K)$ is mapped onto a collinear point.
\end{lemma}

\begin{proof}
Let, for a contradiction,  the point $x$ be mapped onto the collinear point. Selecting a line $L$ through $x$ containing points that are special to both $x^\theta$ and $x^{\theta^{-1}}$, we find a point $p$ on $L$ mapped onto an opposite. Then the line $px$ is fixed under $\theta_p$. However, \cref{coruniclass} asserts that the fix diagram of $\theta_p$ on $\Res_\Delta(p)$ is $\mathsf{E_{7;4}}$, and so there are no fixed points, a contradiction.
\end{proof}

Now we consider the displacement of the symps. 

\begin{cor}\label{notadjacentE8}
No symp of $\mathsf{E_{8,8}}(\K)$ is mapped onto an adjacent one.
\end{cor}

\begin{proof}
Suppose for a contradiction that the symp $\xi$ is mapped onto an adjacent symp $\xi^\theta$. Then the $6$-space $U=\xi\cap\xi^\theta$ is mapped onto another $6$-space $U^\theta$ of $\xi^\theta$. Since $U$ and $U^\theta$ belong to the same natural system of generators of $\xi^\theta$, they intersect in at least one point $p^\theta$. Then $p\in U$ and hence, by \cref{nocollinearE8}, $p=p^\theta$. In the residue of $p$, however, we see a pointwise fixed fully embedded dual polar space isomorphic to $\mathsf{C_{3,3}}(\H,\K)$. But \cite[Corollary~6.11]{npvv} asserts that a collineation pointwise fixing such a dual polar space cannot map symps to adjacent ones. This contradiction concludes the proof of the corollary. 
\end{proof}

\begin{lemma}\label{sympfixedE8}
If a point $x$ is mapped onto a symplectic one, then the corresponding symp $\xi$ is fixed. 
\end{lemma}

\begin{proof}
Let $L$ be line through $x$ such that some point of $L$ is collinear to only a line of $\xi$ (hence $L$ is coplanar with some unique line of $\xi$).  Suppose, for a contradiction, that $L^\theta$ is contained in $\xi$. At least one point $y^\theta$ of $L^\theta$ distinct from $x^\theta$ is not collinear to all points of the unique line $M$ through $x$ coplanar to $L$. Then $y$ and $y^\theta$ are special, contradicting  \cref{nospecialE8}. Hence, by \cref{lemmaE78}, $\xi$ is, in $\Res_\Delta(x)$, adjacent or equal to $\xi^{\theta^{-1}}$.  \cref{notadjacentE8} completes the proof of the lemma, noting that two symps in a point residual of $\Delta$ are adjacent if, and only if, they are adjacent in $\Delta$.
\end{proof}

\subsection{Automorphisms with restricted displacement spectra}\label{sec:biclass}

Our results on domestic automorphisms of spherical buildings have been successfully applied in \cite{Ney-Par-Mal:23} to classify all uniclass automomorphisms of Moufang spherical buildings that are not anisotropic.  A natural question that arises is whether there are more automorphisms having a similar restricted displacement. Specifically, rather than restricting to displacements in a unique conjugacy class (the uniclass condition), one can consider displacements in a restricted number of conjugacy classes of the Weyl group. Of course, if the longest elements is one of these classes, then the automorphism is not domestic, and so our results on domesticity will not be of help. So assume no chamber is mapped onto an opposite. Let us also restrict to type preserving automorphisms. If no chamber is fixed, then from our classifications of domestic automorphisms we see that ``most'' domestic automorphisms are uniclass, and the few that are not appear to have more than two classes in their displacement spectra (see for instance  \cref{equatorexamples} below). Hence we are lead to consider domestic automorphisms that fix at least one chamber and have, besides the identity, displacements in exactly one other conjugacy class of the Weyl group. This seems to be a worthwhile classification, since the result is somewhat unpredictable and nontrivial.

\begin{cor}
Let $\theta$ be a nontrivial type preserving automorphism of a Moufang spherical building $\Omega$ of rank at least $2$ with the property that $\theta$ fixes a chamber and $\mathsf{disp}(\theta)$ is contained in at most $2$ conjugacy classes in the Weyl group. Then $\disp(\theta)=\{1\}\cup\Cl(s)$ for some simple generator $s$ of the Weyl group. Moreover, $\theta$ is either a long root elation, or one of the following cases occurs (and, conversely, each such case does give rise to such an automorphism). In each case of a non-simply laced diagram, we mention, with obvious notation, the class of the Weyl group belonging to the displacement (using standard Bourbaki labelling).
\begin{compactenum}[$(i)$]
\item $\theta$ is a homology in a projective space fixing a hyperplane pointwise (of dimension at least $2$);
\item $\theta$ pointwise fixes a Baer subplane in a (Moufang) projective plane;
\item $\theta$ pointwise fixes a geometric hyperplane in a polar space of rank $r$  at least $2$ (up to duality if the rank is equal to $2$)---class $\Cl(s_r)$;
\item $\theta$ pointwise fixes an ideal Baer sub polar space of a (thick) polar space of rank $3$---class $\Cl(s_1)$;
\item  $\theta$ is an involution in a symplectic polar space pointwise fixing a hyperbolic line and the perp of it, or, equivalently, a imaginary line and its equator geometry---class $\Cl(s_1)$;
\item $\theta$ is an involution pointwise fixing an extended equator geometry and its tropics geometry in the corresponding geometry of type $\mathsf{F_{4,4}}$, or, equivalently pointwise fixing a weak subbuilding with thick frame of type $\mathsf{B_4}$, or, equivalently,  a homology $h_{\omega_4}(-1)$ in Chevalley notation, in a split building of type $\mathsf{F_4}$---class $\Cl(s_4)$;
\item $\theta$ is an involution in a building of relative type $\mathsf{F_4}$ and absolute type $\mathsf{E_6}$ pointwise fixing a corresponding split subbuilding of type $\mathsf{F_4}$---class $\Cl(s_4)$;
\item $\theta$ is a collineation of order $3$ and fixes a subhexagon $\Gamma$, where either $\Omega$ corresponds to a split Cayley hexagon and $\Gamma$ to a non-thick ideal subhexagon (this occurs when the underlying field has nontrivial cubic roots of unity), or $\Omega$ corresponds to a triality hexagon (Moufang hexagon of type $\mathsf{^3D_4}$) and $\Gamma$ to an ideal split Cayley subhexagon---class $\Cl(s_2)$.
\end{compactenum} 
\end{cor}

\begin{proof}
Suppose that $\theta$ is nontrivial, fixes a chamber, and has displacement spectra contained in at most $2$ conjugacy classes of the Weyl group. Clearly $\theta$ maps some chamber to an adjacent chamber, say of Weyl distance~$s$. Thus $s\in\disp(\theta)$, and so $\disp(\theta)\subseteq\{1\}\cup\Cl(s)$. It follows from \cite[Proposition~2.1]{Ney-Par-Mal:23} that $\mathsf{disp}(\theta)=\{1\}\cup\Cl(s)$. We claim that this implies that the opposition diagram of $\theta$ has only one orbit encircled. 

To prove this, we first make the following observations. Let $\Phi$ be an irreducible reduced crystallographic root system with positive system $\Phi^+$ and highest root $\varphi$. Let $\wp$ be the polar type (see \cite[Section~1.1]{PVMexc}). We have $s_{\varphi}=w_{S\backslash\wp}w_0$ (to prove this, note that the inversion sets of both sides are equal, using properties of the highest root and polar type; see \cite{PVMexc}). 

Now, associate a crystallographic system $\Phi$ to the building $\Omega$, choosing between duality classes in such a way that $s=s_{\alpha_i}$ for a long simple root $\alpha_i\in\Phi$. With this choice, the longest member of $\disp(\theta)=\{1\}\cup\Cl(s)$ is $s_{\varphi}$. On the other hand, if the opposition diagram of $\theta$ has nodes $J$ encircled, then the longest member of $\disp(\theta)$ is $w_{S\backslash J}w_0$ (see \cite[Theorem~2.6]{PVMclass}). Thus $s_{\varphi}=w_{S\backslash J}w_0$. But as noted above we have $s_{\varphi}=w_{S\backslash\wp}w_0$, and so $J=\wp$, and hence the claim.

Consider the $\mathsf{A}_n$ case. If $n\geq 3$, then the opposition diagram is $\mathsf{^2A}_{n;1}^1$ and by \cite[Theorem~4.3]{TTV1}, $\theta$ is either a central elation, or a central homology. Now let $n=2$. Then the conjugacy class in the Weyl group $W(\mathsf{A_2})\cong\Sym(3)=\<s_1,s_2\>$ of a generator $s_1$ is $\{s_1,s_2,s_1s_2s_1\}$, which contains the longest word and confirms that $\theta$ is not (necessarily) domestic. Hence a `forbidden' distance is $s_1s_2$, which means that, if we map a point $x$ to different point $x^\theta$, then the line $\<x,x^\theta\>$ must be fixed, as otherwise the chamber $\{x,
\<x,x^\theta\>\}$ has distance $s_1s_2$ to $\{x^\theta,\<x,x^\theta\>^\theta\}$. By \cite[Proposition~3.3]{PVMclass}, $\theta$ is either a central long root elation, a central homology, or a Baer collineation, and each these also satisfy the dual condition that the point of intersection of a line and its (distinct) image is fixed (displacement distance $s_2s_1$ is not attained).  This concludes the $\mathsf{A}_n$ case. 

Consider the $\mathsf{B}_n$ case, $n\geq 2$. Here the opposition diagram is never full, and is one of $\mathsf{B}_{n;1}^1$ or $\mathsf{B}_{n;1}^2$. So $\theta$ is domestic, and we can apply the results of \cite{PVMclass,TTVM2}. First suppose $n=2$. Then \cite[Theorem~2.1]{TTVM2} implies the assertion (noting tat a \emph{large full subquadrangle} is a geometric hyperplane) and it is easy to check that in each case $\disp(\theta)$ is really the union of the identity and one more class (the cases of an ovoid or spread elementwise fixed corresponds to uniclass collineations).  

Now suppose $n\geq 3$. If the opposition diagram is $\mathsf{B}_{n;1}^1$, then \cite[Theorem~1]{PVMclass} implies that $\theta$ is a symmetry (following Dieudonn\'e \cite{Die:73}, we call a collineation of a polar space pointwise fixing a geometric hyperplane, a \emph{symmetry}; these include central elations). Clearly all symmetries have the trivial and only one non-trivial class in their displacement spectra. Now assume the opposition diagram is $\mathsf{B}_{n;1}^2$. Then \cite[Theorem~3]{PVMclass} implies that $\theta$ is either an axial long root elation, a Baer collineation in rank~$3$, or a (generalised) homology in a symplectic polar space in characteristic different from~$2$ pointwise fixing a hyperbolic line and its perp. We check whether $\disp(\theta)$, for $\theta$ a Baer collineations in rank~$3$ (which is a point-domestic collineation), consists of two conjugacy classes among which the trivial one. We leave the other (easier) cases to the reader. The technique of doing so is similar to the proofs in \cite{Ney-Par-Mal:23} that establish certain automorphisms to be uniclass. Call the generators of the Weyl group $s_1,s_2,s_3$ with standard Bourbaki labelling. Let $C$ be a chamber, which we will represent by a triple $\{p,L,\pi\}$, with $p$ a point, $L$ a line and $\pi$ a plane of the polar space $\Omega$ of rank 3. Recall from \cite[\S7]{TTVM} that a Baer collineation is a collineation with fix stucture an ideal Baer sub polar space, which is a subset $B$ of points of $\Omega$ with the properties that in each plane $\alpha$ which contains at least two points of   $B$, the set $B$ induces a Baer subplane of $\alpha$, and $B$, together with all secants (that is, lines of $\Omega$ intersecting $B$ in at least two points) form a polar space of rank $3$. Now suppose first that $p\in B$. If $L$ is a secant, then both $L$ and $\pi$  are fixed, leading to trivial displacement. Suppose $L$ is not fixed. Then either $\pi$ is fixed (and the displacement is $s_2\in\Cl(s_1)$), or $\pi^\theta$ intersects $\pi$ in exactly $p$ (as otherwise a point of the intersection is collinear to $L$ and $L^\theta$, hence to a plane to which it does not belong, a contradiction to the rank).  Then the displacement of $\theta$ is  $s_3s_2s_3\in\Cl(s_1)$. Now suppose that $p$ is not fixed. Then $p^\theta$ is collinear to $p$. If $p^\theta\in L$, then $\pi$ is fixed and the displacement is $s_1\in\Cl(s_1)$. Since all planes through $p$ and $p^\theta$ are stabilised, the line $L$ is mapped onto a collinear one, and as before $\pi$ is either fixed (leading to the displacement $s_2s_1s_2\in\Cl(s_1)$) or $\pi\cap\pi^\theta$ is the single point $L\cap L^\theta$, in which case the displacement is $s_3s_2s_1s_2s_3\in\Cl(s_1)$. This completes the proof for the Baer collineations.

Consider the $\mathsf{D}_n$ case, $n\geq 4$. Then $\theta$ has opposition diagram $\mathsf{D}_{n;1}^1$ or $\mathsf{D}_{n;1}^2$. In the former case, \cite[Theorem~2.6]{PVMclass} asserts that $\theta$ is a type rotating involution with fix structure a split building of type $\mathsf{B}_{n-1}$ (and which is a geometric hyperplane), however since this automorphism is not type preserving it is not considered here. In the latter case \cite[Theorem~3(2)]{PVMclass} implies that $\theta$ is a long root elation. 

Consider the $\mathsf{E}_n$ case, $n\in\{6,7,8\}$. Them $\theta$ has opposition diagram $\mathsf{^2E_{6;1},E_{7;1}}$ or $\mathsf{E_{8;1}}$. By \cite[Theorem~1]{PVMexc}, $\theta$ is a long root elation. 

Consider the $\mathsf{F}_4$ case. If $s\in\{s_1,s_2\}$ then we have opposition diagram $\mathsf{F}_{4,1}^1$, and it is a long root elation. If $s\in \{s_3,s_4\}$ then we have opposition diagram $\mathsf{F}_{4,4}^1$. Such automorphisms are classified in~\cite{Lam-Mal:24}, they are either the homology $h_{\omega_4}(-1)$ (in the split case), or a diagram automorphism $\sigma$ of $\sE_6$ (in a twisted case). In either case they are involutions. So by \cite[Proposition~2.1]{Ney-Par-Mal:23} their displacements are unions of full conjugacy classes of involutions in the Weyl group. Now, by a direct calculation in the Weyl group of type $\sF_4$, all classes of involutions have elements longer than $15$ except for the classes of $s_1$ and $s_4$ (see \cite[Theorem~1.9]{Ney-Par-Mal:23} for a description of the involution classes). Since the numerical displacement of our automorphism is $15$, it forces it to only contain these classes (with the identity) -- but it does not contain the class of $s_1$ (otherwise there is a type $1$ vertex mapped to an opposite). So it indeed has the correct displacement. 

Consider the $\mathsf{G_2}$ case.  According to \cite[Theorem~1]{PVMexc2}, the collineations with opposition diagram $\mathsf{G_{2;1}^1}$ and $\mathsf{G_{2;1}^2}$ fixing a chamber occur precisely if $\theta$ is either a (central or axial) long root elation, or $\theta$ is a collineation of order $3$ and fixes a subhexagon $\Gamma$, where either $\Omega$ corresponds to a split Cayley hexagon and $\Gamma$ to a non-thick ideal subhexagon (this occurs when the underlying field has nontrivial cubic roots of unity), or $\Omega$ corresponds to a triality hexagon (Moufang hexagon of type $\mathsf{^3D_4}$) and $\Gamma$ to an ideal split Cayley subhexagon. In all these case, it is easily checked that $\disp(\theta)$ is the union of two conjugacy classes, one of which is the identity.  

Finally consider the $\mathsf{^2F_4}$ case, that is, the case of Moufang octagons. We appeal to Theorems~2.7 and~2.8 of \cite{PTV:15} to see that $\theta$ either fixes a large suboctagon, or $\theta$ is a long root elation.  Proposition~4.3 and \S5.1 of \cite{Jos-Mal:95} imply that $\Omega$ has a unique class of full or ideal suboctagons, and Proposition~5.1 of that paper proves that no such suboctagon is pointwise fixed by a nontrivial collineation.  
\end{proof}

\begin{rem}\label{equatorexamples}
We will now show that the displacement set of a Class II automorphism of a large $\sE_8$ building contains at least three elements from distinct classes in the Weyl group (and precisely the same argument applies to the automorphisms of $\sE_7$ buildings from \cite[Theorem~1(iii)]{npvv}). If $\theta$ is of Class II then by Theorem~\ref{thm:equatordomestic} $\theta$ is conjugate to an element of the form $x_{\alpha_4}(a)x_{-\alpha_4}(1)$ for some $a\in\K$ with $aX^2+aX-1$ irreducible over~$\K$ (in Bourbaki labelling). This element preserves the the residues of the base chamber of types $\{4\}$, $\{3,4,5\}$, and $\{2,3,4,5\}$ (these residues are buildings of types $\sA_1$, $\sA_3$, and $\sD_4$ respectively). By Remark~\ref{rem:highrank} the restriction of $\theta$ to these residues is not domestic, and hence $\disp(\theta)$ contains the elements $s_4,w_{\{3,4,5\}}$, and $w_{\sD_4}$. These elements all lie in different conjugacy classes in the Weyl group (indeed, $w_{\sD_4}$ is of minimal length in its conjugacy class by \cite[Theorem~1.9]{Ney-Par-Mal:23} and so it is not conjugate to either $s_4$ or $w_{\{3,4,5\}}$, and $s_4$ and $w_{\{3,4,5\}}$ are not conjugate by parity of length). 
\end{rem}

\textbf{Addresses of the authors}\\

James Parkinson\\
School of Mathematics and Statistics, The University of Sydney, Carslaw Building, F07, Sydney, NSW 2006, Australia\\
\texttt{jamesp@maths.usyd.edu.au}

Hendrik Van Maldeghem\\
Department of Mathematics, Computer Science, and Statistics, Ghent University, Krijgs\-laan 281-S9, 9000 Ghent, Belgium\\
\texttt{Hendrik.VanMaldeghem@UGent.be}

\end{document}